\documentclass[12pt]{amsart}
  \usepackage{amsfonts,amsmath,amssymb,amstext}
\usepackage{hyperref}
\usepackage{comment}
\usepackage{pst-node}

\usepackage{tikz-cd} 
\usepackage{color}

\usepackage{tikz}
\usepackage{comment}

  \usepackage[letterpaper,margin=2cm,includeheadfoot]{geometry}
  
\newcounter{magicrownumbers}
\newcommand\rownumber{\refstepcounter{magicrownumbers}\arabic{magicrownumbers}}

\newtheorem{thmm}{Theorem}[section]

\newtheorem{remark}[thmm]{Remark}
\newtheorem{example}[thmm]{Example}
\newtheorem{definition}[thmm]{Definition}
\newtheorem{proposition}[thmm]{Proposition}
\newtheorem{lemma}[thmm]{Lemma}
\newtheorem{corollary}[thmm]{Corollary}
\newtheorem*{thmA*}{Theorem A}
\newtheorem*{thmB*}{Theorem B}
\newtheorem*{thmC*}{Theorem C}

\begin{document}
\bibliographystyle{amsplain}

\relax
\renewcommand{\v}{\varepsilon} \newcommand{\p}{\rho}
\newcommand{\m}{\mu}
\def\im{{\rm Im}}
\def\ker{{\rm Ker}}
\def\End{{\rm End}}
\def\Pic{{\bf Pic}}
\def\re{{\bf re}}
\def\e{{\bf e}}
\def\a{\alpha}
\def\ve{\varepsilon}
\def\b{\beta}
\def\D{\Delta}
\def\d{\delta}
\def\f{{\varphi}}
\def\ga{{\gamma}}
\def\L{\Lambda}
\def\lo{{\bf l}}
\def\s{{\bf s}}
\def\A{{\bf A}}
\def\B{{\bf B}}
\def\cB{{\mathcal {B}}}
\def\C{{\mathbb C}}
\def\F{{\bf F}}
\def\G{{\mathfrak {G}}}
\def\g{{\mathfrak {g}}}
\def\aa{{\mathfrak {a}}}
\def\b{{\mathfrak {b}}}
\def\q{{\mathfrak {q}}}
\def\f{{\mathfrak {f}}}
\def\k{{\mathfrak {k}}}
\def\l{{\mathfrak {l}}}
\def\m{{\mathfrak {m}}}
\def\n{{\mathfrak {n}}}
\def\o{{\mathfrak {o}}}
\def\p{{\mathfrak {p}}}
\def\s{{\mathfrak {s}}}
\def\t{{\mathfrak {t}}}
\def\r{{\mathfrak {r}}}
\def\z{{\mathfrak {z}}}
\def\h{{\mathfrak {h}}}
\def\H{{\mathcal {H}}}
\def\O{\Omega}
\def\M{{\mathcal {M}}}
\def\T{{\mathcal {T}}}
\def\N{{\mathcal {N}}}
\def\U{{\mathcal {U}}}
\def\Z{{\mathbb Z}}
\def\P{{\mathcal {P}}}
\def\GVM{ GVM }
\def\iff{ if and only if  }
\def\add{{\rm add}}
\def\deg{{\rm deg}}
\def\Hom{{\rm Hom}}
\def\ld{\ldots}
\def\vd{\vdots}
\def\sl{{\rm sl}}
\def\mod{{\rm mod}}
\def\len{{\rm len}}
\def\cd{\cdot}
\def\dd{\ddots}
\def\q{\quad}
\def\qq{\qquad}
\def\ol{\overline}
\def\tl{\tilde}
\def\nn{\nonumber}
\def\eps{\varepsilon}
\def\beps{\bar\eps}
\def\bdelta{\bar\delta}
\def\blue{\color{black}}
\def\black{\color{black}}
\def\red{\color{black}}
\def\gr{\operatorname{gr}}
\def\Ita{I_\theta(\red \aa^*\black)}

\title {Restriction theorems and root systems for symmetric superspaces}

\author[]{Shifra Reif, Siddhartha Sahi and Vera Serganova}
\date{ \today }

\maketitle

\begin{abstract}
    In this paper we consider those involutions $\theta$ of a finite-dimensional Kac-Moody Lie superalgebra $\g$, with associated decomposition $\g=\k\oplus\p$, for which a Cartan subspace $\aa$ in $\p_{\bar 0}$ is self-centralizing in $\p$. For such $\theta$ the restriction map $C_\theta$ from $\p$ to $\aa$ is injective on the algebra $P(\p)^\k$ of $\k$-invariant polynomials on $\p$. There are five infinite families and five exceptional cases of such involutions, and for each case we explicitly determine the structure of $P(\p)^\k$ by giving a complete set of generators for the image of $C_\theta$. We also determine precisely when the restriction map $R_\theta$ from $P(\g)^\g$ to $P(\p)^\k$ is surjective. Finally we introduce the notion of a generalized restricted root system, and show that in the present setting the $\aa$-roots $\Delta(\aa,\g)$ always form such a system. 
    
\end{abstract}

\section{Introduction}

In this paper we consider the analog of the Chevalley restriction theorem for an involution $\theta$ of a finite dimensional Kac-Moody Lie superalgebra $\g=\g_{\bar 0}+\g_{\bar 1}$. That is to say we study the structure of the algebra $P(\p)^\k$ of $\k$-invariant polynomials on $\p$, where $\k$ is the fixed subalgebra of $\theta$ and $\p$ is the $(-1)$-eigenspace, by analyzing its restriction to a Cartan subspace $\aa$ in $\p_{\bar 0}=\p\cap\g_{\bar 0}$. This defines a map $C_\theta$ from $P(\p)^\k$ to the polynomial algebra $P(\aa)$, which is injective precisely when the centralizer of $\aa$ in $\p_1=\p\cap\g_{\bar 1}$ is $0$. This is equivalent to $\g$ admitting an Iwasawa decomposition of the form $\g=\k\oplus\aa\oplus \n$, and so we will say such a $\theta$ is an Iwasawa involution. In this case $P(\p)^\k$ is isomorphic to the image of $C_\theta$. In the rest of the paper we assume that $\theta$ is an Iwasawa involution.

In Theorem A we determine the image of $C_\theta$ for an Iwasawa involution $\theta$. More precisely we show that the image is equal to a certain subalgebra $\Ita$ of $P(\aa)$ defined below. Our result is completely explicit: there are five infinite families of Iwasawa involutions and five exceptional cases, and in each case we provide a complete set of generators for $\Ita$ except for one subfamily (see Section \ref{sec:generators table}).

We consider also the algebra $P(\g)^\g$ of $\g$-invariant polynomials on $\g$ and its restriction to $\p$. This defines a map $R_\theta$ from $P(\g)^\g$ to $P(\p)^\k$  which is in general neither injective  nor surjective. In Theorem B we show that $R_\theta$ is in fact surjective \emph{except} for a subfamily of one infinite family and two of the exceptional cases. Again the results are completely explicit: we determine the image of the composite restriction map $C_\theta \circ R_\theta$ and show that it coincides with $\Ita$ except in these cases.

We also determine the structure of the set $\Delta(\aa,\g)$ of roots of $\aa$ in $\g$. 
Once more we proceed in an explicit manner: we choose a Cartan subalgebra $\h$ of $\g$ which contains $\aa$, and we study the restriction of the $\h$-roots to $\aa$. 
The set $\Delta(\h,\g)$ of $\h$-roots is a generalized root system (GR system)  in the sense of Serganova [Serga2], however this is not always the case for $\Delta(\aa,\g)$. 
Therefore \blue it is natural to \black introduce the notion of a generalized restricted root system (GRR system) by relaxing one of the requirements of a GR system. 
\blue We show that $\Delta(\aa,\g)$ is always a GRR system  in Proposition \ref{GRR prop} (see Section 2.2 for the precise definition).\black

We now describe our results more precisely, starting with the definition of $\Ita$. Let $W$ be the Weyl group of $\Delta_0=\Delta(\aa,\g_{\bar 0})$ and let $P(\aa)^W$ be the algebra of $W$-invariant polynomials on $\aa$. We say that a root $\a\in\Delta(\aa,\g)$ is \emph{singular} if no multiple of $\alpha$ is a root in $\Delta_0$. The root space $\g_\a$ of a singular root is purely odd, and in fact $\g_\a$ has dimension $\left(0| 2k_\a\right)$ for some integer $k_\a$. We define $\Ita$ to be the subalgebra consisting of $W$-invariant polynomials $f \in P(\aa)^W$ which satisfy the following derivative condition for every singular $\a$
\[ \left(D_{h_{\alpha}}\right)^{k}\left(f\right)\in\langle \alpha\rangle, \quad k=1,3,\ldots,2k_\alpha-1  .\]
Here $h_\a$ is the image of $\a$ under the map $\aa^* \to \aa$ induced by an invariant bilinear form, $D_{h_{\alpha}}$ is the partial derivative along $h_\a$, and $\langle \a \rangle$ is the principal ideal of $P(\aa)$ generated by $\a$.

\smallskip

The main results of this paper, which hold in the setting of an Iwasawa involution $\theta$ of a \red  finite-dimensional \black Kac-Moody Lie superalgera $\g$, are as follows.
 
\begin{thmA*} \label{chev rest thm} The restriction map $C_\theta: P(\p)^\k\rightarrow P(\aa)$ is injective and its image is $\Ita$.
\end{thmA*}

\begin{thmB*}

\label{thm:surjection of P(g)^g}
     The restriction map $R_\theta : P(\g)^\g\rightarrow P(\p)^\k$ is surjective unless $(\g,\k)$ is one of the pairs $\left(\mathfrak{osp}(2m|2n_1+2n_2),\mathfrak{osp}(m|2n_1)\oplus\mathfrak{osp}(m|2n_2)\right)$, $(F_4,\mathfrak{gosp}(2|4))$ or $(D(2,1,a),\mathfrak{osp}(2|2)\oplus \mathfrak{so}(2))$. 
\end{thmB*}
  

    \smallskip
    
    Theorem A was previously obtained in \cite{AHZ} in a slightly different setup, with a completely different proof. Our argument is very explicit: we directly verify that Theorem A holds for every Iwasawa involution. This approach allows us to prove Theorems B and reveals a great deal of information about Iwasawa involutions, generalized restricted root systems, and root multiplicities.
    We hope that this information will be useful for researchers in the general theory of Lie superalgebras, as well as those interested in the applications of Lie superalgebras. We briefly describe one specific application that we have in mind.
     
    In \cite{SV} Sergeev and Veselov have developed the theory of a deformed GR system, which is a GR system together with a bilinear form and a $W$-invariant multiplicity function on the roots satisfying certain ``admissibility" constraints. They have used this theory to construct families of completely integrable systems, which are supersymmetric analogs of the classical Calogero-Moser-Sutherland models of mathematical physics. It is natural to ask whether there exists a similar deformation theory for a GRR system, which might lead to new integrable systems. In a sequel to this paper we intend to study this question in conjunction with the classification problem for GRR systems, and we expect that the results of the present paper will play a key role in these investigations.

    The paper is organized as follows. In Section 2 we recall some generalities on roots and restricted roots for a Lie superalgebra, \blue  give the definition of a GRR system and in Proposition \ref{GRR prop} prove that $\Delta(\aa,\g)$ is GRR system. We also compare GRR systems to a similar concept (RGRS) introduced in \cite{Sh2}. \black In Sections 3 and 4, we show that $C_\theta$ is injective and that its image is contained in $\Ita$. Finally in Section 5, we show by explicit computation that the image of $C_\theta$ is precisely $\Ita$. These computations also let us prove \red  Theorems A and B. \black  Our analysis reveals a fair amount of detailed information about the sets of restricted roots, and we organize this in the appendix for future reference.

\subsection*{Acknowledgements}
    The project was made possible by a SQuaRE at the American Institute for Mathematics. The authors thank AIM for providing a supportive and mathematically rich environment."
    This project is partially supported by Israel Science Foundations Grant 1957/21.
    S. Reif gratefully acknowledges support from the Institute for Advanced Study while working on this project.
    The research of S. Sahi was partially supported by NSF grants DMS-1939600 and 2001537, and Simons Foundation grant 509766.
    Vera Serganova was supported in part
by NSF grant 2001191 and by Tromso research foundation.

\section{Supersymmetric spaces}
\red We assume in this paper that the base field is $\mathbb C$.\black
\subsection{Restricted Root Systems}\label{restricted root systems prelim}
Let $\g=\g_{\bar 0}\oplus\g_{\bar 1}$ be a \red  finite-dimensional \black  Kac-Moody Lie superalgebra with a non-degenerate even bilinear form $(\cdot,\cdot)$, namely $\g$ is one of $\mathfrak{gl}(m|n)$, \blue $\mathfrak{sl}(m|n)_{m\ne n}$, $\mathfrak{psl}(n|n)$,  \black $\mathfrak{osp}(m|2n)$, $D(2,1,a)$, $F_4$ and $G_3$. Let $\theta$ be an involution on $\g$ which fixes $(\cdot,\cdot)$. 
Then $\g=\k\oplus\p$ where $\k$ is the subalgebra of $\g$ of $\theta$-fixed points and $\p:=\{x\in\g:\theta(x)=-x\}$. 
Let $\aa\subseteq \p_{\bar 0}$ be the Cartan subspace of the symmetric pair $(\g_{\bar 0},\k_{\bar 0})$, \red that is $\aa$ is a commutative and $\p_{\bar 0}^\aa=\aa$. \black
We assume that there exists an Iwasawa decomposition  $\g=\k\oplus\aa\oplus\n$ where $\n:=\operatorname{span}_{\mathbb C}\{x\in \g: [a,x]=c_a x, c_a\in\mathbb R_{>0}\} $ for some regular element $a\in\aa$.  
This is equivalent to $\p^\aa=\aa$. 
Let  $\h$ be a Cartan subalgebra of $\g$ containing $\aa$ and let $\t:=\h\cap\k$.

\red Note that $\h$ is $\theta$-stable. Indeed, let $h\in\h$ and write $h=h_k+h_p$ where $h_k\in\k$ and $h_p\in\p$. Then for any $a\in\aa$, $0=[a,h]=[a,h_k]+[a,h_p]$. Since $[a,h_k]\in\p$ and $[a,h_p]\in\k$, $[a,h_k]=[a,h_p]=0$. Hence $h_p\in\aa\subseteq\h$ and so $h_k\in\h$ and $\theta(h)=h_k-h_p\in\h$. \black
We denote by $p_\aa:\h\rightarrow\aa$ the projection with kernel $\t$. Since $\theta$ preserves $(\cdot,\cdot)$, it restricts to a nondegenerate invariant form on $\aa$ and $\t$.

We denote by $\Delta(\g,\h)$ the set of roots of $\g$ and define the set of restricted roots to be
$\Delta(\red \aa,\g\black ):=\{\bar\alpha\mid_\aa\ : \bar\alpha\in\Delta(\g,\h)\}\subset \aa^*$. 
For a restricted root $\alpha$, the root space is defined by $\g_\alpha:=\{x\in\g : [h,x]=\alpha(h)x,\ \forall h\in\aa\}$ and $\dim \g_\a$ is called the \emph{multiplicity} of $\a$.
The root space $\g_\a$ need not be purely even or odd space. 
A root $\a$ is called \emph{singular} if $\g_{c\a}\cap \g_{\bar 0}=\{0\}$ \red  for any \black  $c\in\mathbb C$.
Note that in $\Delta(\g,\h)$ \red a root does not admit a multiple in $\Delta_0$ \black if and only if it is isotropic. Hence all singular roots are \red  restrictions \black  of isotropic roots.

Since the restriction of $(\cdot,\cdot)$ to $\aa$ is nondengenerate, 
we have a form on $\aa^*$, \red which we again define denote by $(\cdot,\cdot)$. \black We define the baby Weyl group to be $W=\langle s_\alpha \mid \alpha\in\Delta(\red \aa,\g\black )_{\bar 0}\rangle$ where $\Delta(\red \aa,\g\black )_{\bar 0}$ is the set of roots which are restrictions of even roots, and $s_\alpha$ is the reflection corresponding to $\alpha$.
Choose  a nonzero vector $h_\alpha$  in $(\ker\alpha)^\perp\subset \aa$. Note that $(\ker\alpha)^\perp$ is one dimensional so $h_\alpha$ is unique up to a scalar.
\begin{remark}
The projection $p_\aa([\g_\alpha,\g_{-\alpha}])$ is orthogonal to $\ker\a$. Indeed, letting $x\in\g_\alpha$ and $y\in\g_{-\alpha}$, since $(\aa,\t)=0$, we have 
\[(p_\aa([x,y]),a)=([x,y],a)=(x,[y,a])=\alpha(a)(x,y)=0\]
for all $a\in \ker\alpha$.
\end{remark}

Let $B_\theta(x,y)=(x,\theta y)$ be a twisted form defined on $\g$. Then $B_\theta$ is a supersymmetric even non-degenerate bilinear form, see \cite[Prop. 2.10]{AHZ}. Note that for every root space $\theta\g_\a=\g_{-\a}$ because $[h,\theta x]=-\theta([h,x])=-\alpha(h)\theta x$ for any $h\in\aa,x\in \g_\a$.  Hence $B_\theta$ restricts to a non-degenerate supersymmetric form on $\g_\a$.
\red
\begin{lemma}\black
\label{skew sym form}

    If $\alpha$ is a singular root, $B_\theta$ is a \red  non-degenerate \black  skew-symmetric bilinear form on $\g_\alpha$. 
    In particular $\dim \g_\a$ is even. Moreover, if $B_\theta(x,y)\ne 0$ \red  for $x,y\in\g_\a$   then $p_\aa[x,\theta y]$ is a nonzero multiple of $h_\a$.\black
    \red 
\end{lemma}
\red
    \begin{proof}
    If $\a$ is singular then $\g_\a$ is purely odd and so $B_\theta$ is skew-symmetric.
The subspaces $\aa$ and $\mathfrak t$ are orthogonal complements with respect to $(\cdot,\cdot)$ since  $(\cdot,\cdot)$ is $\theta$ stable. Hence $(p_\aa[x,\theta y], a)=([x,\theta y], a)$ for any $a\in \aa$.    
Moreover, \black
$$([x,\theta y], a)=-([\theta x,y],-a)=(\theta x,[y,a])=-\alpha(a)(x,\theta y)=-\alpha(a) B_\theta(x,y).$$
Since $B_\theta(x,y)\ne 0$, we get that $p_\aa[x,\theta y]$ is a nonzero multiple of $h_\a$.
     
    \end{proof}
\black

\subsection{The Symmetric Algebra and the Algebra of Invariants}
For a super-vector space $V=V_{\bar 0}\oplus V_{\bar 1}$, \red  we denote by 
$T(V)$ the tensor algebra, namely $T(V)=\bigoplus_{d=0}^\infty V^{\otimes d}$. Similarly, \black  we denote by $S(V)$ the symmetric algebra on $V$, that is $S=S(V_{\bar 0})\otimes \bigwedge(V_{\bar 1})$ as vector spaces. 
Note that $S(V)=\bigoplus_{d=0}^\infty S^d(V)$ where $S^d(V)$ is the component of degree $d$. 

Let $m=\dim V_{\bar 0}$ and $n=\dim V_{\bar 1}$. Denote by $P_d$ be the set of partitions $\lambda$ of $d$ that fit in the $(m|n)$-hook, namely $\lambda_{m+1}\le n$. 
Recall that $\red  V^{\otimes d}\black =\bigoplus_{\lambda\in P_d} V_\lambda$ as a $\mathfrak{gl}(V)$-module where  $V_\lambda$ is the irreducible module with highest weight 
$$\lambda_1\eps_1+\ldots+\lambda_m\eps_m+\max(\lambda'_1-m,0)\delta_1+\ldots+\max(\lambda'_n-m,0)\delta_n.$$

Note that for any superspace $V$, the polynomial algebra $P(V)$ is naturally isomorphic to the symmetric algebra $S(V^*)$. Given a subspace $W\subset V$, the restriction map from $S(V^*)$ to $S(W^*)$ is well defined.

    Given an Iwasawa involution $\theta$, the algebra  of invariants is the following subalgebra of $S(\aa^*)$:
    $$\Ita=\left\{f\in S(\aa^*)^W \mid \left(D_{h_{\alpha}}\right)^{k}\left(f\right)\in\left\langle \alpha\right\rangle \text{ for } \a \text{ singular }, k=1,3,\ldots,\dim\g_\a-1  \right\}.$$
    We shall abbreviate $I(\aa^*)=\Ita$ in the rest of the paper.

\begin{example} \label{diagonal case}
    Let $\tilde \g:=\g\times\g$ and let $\tilde\theta:\tilde\g\rightarrow\tilde\g$ be such that $\tilde\theta(x,y)=(y,x)$. 
    Then we get a symmetric pair $(\tilde\g,\tilde\k)$ where $\tilde\k=\{(x,x):x\in\g\}\cong\g$. Here $\tilde\p=\{(x,-x):x\in\g\}$
    and the Cartan subspace  $\tilde\aa=\{(h,-h):h\in\h\}$ is identified with the Cartan subalgebra $\h$ of $\g$. 
    In this case the restricted roots are the roots of $\g$, the dimension of each root space is $2$ and each root space is either purely even or purely odd. \red A root of $\g$ with an odd root space has multiple which is an even root if and only if it is nonisotropic. \black Hence the algebra of invariants is $$I(\tilde \aa^*)\cong I(\h^*):=\left\{ f\in S(\aa^*)^W : D_{h_\a}(f)\in\langle\a\rangle \text{ for $\a$ isotropic root}\right\}.$$
    By \cite{Serge2}, this is precisely the image of the restriction map $S(\g^*)^\g$ to $S(\h^*)$.
\end{example}

\begin{remark} \label{surjectivity via gr}
    The algebra $S(\mathfrak g)$ is in fact  $\gr U(\g)$, namely the associated graded of the universal enveloping algebra of $\g$. 
    Since the adjoint action of $\g$ gives an isomorphic $\g$-module structures on  $U(\g)$ and $S(\g)$, we also get that $\gr Z(\g)=S(\g)^\g$. 
    Let $D(\g,\k):= U(\g)^\k / \red(\k U(\g)\cap U(\g)^\k)\black$. 
    Since $\g=\k\oplus \p$ as a $\k$-module, we also get that $\gr D(\g,\k)=S(\p)^\k\cong S(\g)^\k/ (\k S(\g)\cap S(\g)^\k)$. The functor $\gr$ sends the projection map $Z(\g)\rightarrow D(\g,\k)$ to the projection $S(\g)^\g\rightarrow S(\p)^\k$ and reflects surjectivity, that is the former map is surjective if the latter map is surjective. Using the invariant bilinear form, we note that the latter map is equivalent to the restriction map $S(\g^*)^\g\rightarrow S(\p^*)^\k$.
\end{remark}

\subsection{Generalized Restricted Root Systems}

We show that a restricted root system is a generalized root system in the following sense.
\begin{definition}\label{def:grrs}
Let $V$ be a finite-dimensional vector space with a non-degenerate bilinear form $\langle\cdot,\cdot\rangle$. A finite set $R=R_{sing}\cup R_{reg}\subset V\backslash \{0\}$ is called a GRR system if 
\begin{enumerate}
    \item $R$ spans $V$ and $R_{reg}=-R_{reg}$, $R_{sing}=-R_{sing}$.
    \item If $\a\in R_{reg}$ then $\langle \a,\a\rangle \ne 0$. 
For every $\beta\in R$, $\frac{2\langle\a,\beta\rangle}{\langle \a,\a\rangle}\in\mathbb Z$ and $s_\a(\beta)=\beta-\frac{2\langle\a,\beta\rangle}{\langle \a,\a\rangle}\a\in R$. 
Moreover, $s_\alpha(R_{reg})=R_{reg}$ and $s_\a(R_{sing})=R_{sing}$.
    \item If $\a\in R_{sing}$ then for any $\beta\in R$ such that $\langle \a,\beta\rangle\ne0$, at least one of the vectors $\beta+\a,\beta-\a$ belongs to $R$. 
    Moreover, if $\alpha,\beta\in R_{sing}$ and $\alpha+\beta\in R$ then $\alpha+\beta\in R_{reg}$.
\end{enumerate}
\end{definition}

This generalizes the definition of \cite{Serga2}, where in the third condition it is required that exactly one of $\a+\beta,\a-\beta$ belongs to $R$. The notion of GRR system also includes the root system $BC(m,n)$ given in  \cite[Sec. 2]{SV}. Note that $R_{reg}$ is a nonreduced root system.

We now prove that  $R=\Delta(\red \aa,\g\black )$ is a GRR system for every Iwasawa decomposition. 
Here $R_{sing}$ are the singular roots and the inner product is given by the inner product of $\aa$. 
\blue
  \begin{proposition} \label{GRR prop}

      The restricted roots $\Delta(\aa,\g)$ form a GRR system.      
  \end{proposition} \black 

\begin{proof}
    The property (1) follows from the same property of $\Delta(\g,\h)$. 
    \red 
    Let $\a'\in R_{\red reg}$. 
   Then $\a'=k\alpha$ where $\g_{\a}\cap\g_{\bar 0}$ is nonzero for some $k\in\mathbb C$. 
    From the theory of Lie algebras $( \a,\a)\ne 0$ and so $( \a',\a')\ne 0$. 
    \black 
    Moreover, there exists $x\in \g_{\a}\cap\g_{\bar 0}$ such that $(x,\theta x)\ne 0$. 
    Indeed, otherwise for every $y,z\in \g_{\a}\cap\g_{\bar 0}$, one has $(y,\theta z)=\frac{1}{4}\left( (y+z,\theta y+\theta z)-(y-z,\theta y-\theta z)\right)=0$, which contradicts the fact that $(\cdot,\cdot)$ is non degenerate. 
    Now, since $(x,\theta x)\ne 0$, the subalgebra $\mathfrak{s}_\a=\operatorname{span}_{\mathbb R}\{x,\theta x,[x,\theta x]\}$ is isomorphic to $\mathfrak{sl}_2(\mathbb R)$ so (2) follows for $\a$.
    We are left to prove (2) for \red $\a'=k\a$\black . 
    Since $\bigoplus_{r\in\mathbb C}\g_{r\a}$ is a finite-dimensional \red $\mathfrak s_{\a}$\black -module, \red $ \g_{k\a}$ is nonzero only if $k\in \frac{1}{2}\mathbb Z$
    \black 
    and $(2)$ holds. 

    For (3), we first show that  $(\a,\beta)$ is a nonzero multiple of $\beta(h_\a)$ for any $\alpha,\beta\in R$. 
    Let $x,y\in\g_\a$ be such that $(x,\theta y)\ne 0$. Then \red by Lemma \ref{skew sym form}, $H_\a=p_\aa([x,\theta y])$ is a nonzero scalar multiple of $h_\alpha$, \black 
   where $h_\a$ is defined by $(h_\a,h)=\a(h)$ for all $h\in\aa$. 
    Hence $(h_\a,h_\beta)\ne 0$
    if and only if
    $(H_\a,H_\beta)\ne 0$. \black 
    These inequalities are equivalent to
    $\alpha(H_\beta)\ne 0$
    because
    $$(H_\a,H_\beta)=([x,\theta y],H_\beta)=(x,[\theta y,H_\beta])=\a(H_\beta)(x,\theta y)$$
     and $(x,\theta y)\ne 0$. By symmetry this is equivalent to $\beta(H_\a)\ne 0$. Thus, $(\a,\beta)$ is a nonzero multiple of $\beta(h_\a)$.

     Suppose that $a\in \red R_{sing}\black $, $\beta\in R$ but $\beta\pm\a\notin R$.  
     This means that for any $z\in \g_\beta$, $[\g_\a,z]=[\g_{-\a},z]=0$. In particular $[x,z]=[y,z]=[\theta x,z]=[\theta y,z]=0$. This implies that 
     $$0=[[x,\theta y]-\theta([x,\theta y]),z]=2[p_\aa([x,\theta y]),z]=2[H_\a,z]=2\beta(H_\a)z.$$
     Thus $\beta(H_\a)=0$  as required.

     Suppose $\alpha,\beta\in R_{sing}$. Then $\g_\a,\g_\beta$ are purely odd. Since $\g$ is a Kac-Moody, if  $\a+\beta\in R$, then $\g_{\a+\beta}=[\g_\a,\g_\beta]$. Hence $\g_{\a+\beta}$ must contain a nonzero even vector and so $\a+\beta\in R_{reg}$.
     \end{proof}
\blue
\begin{example} \label{GRR Hermitian example}
    Not every GRR system arises from a symmetric pair. For example,
    $$\Delta_{reg}=\left\{\eps_i-\eps_j:1\le i\ne j\le n\right\},\quad \Delta_{sing}=\left\{\pm(\eps_i+\eps_j):1\le i\ne j\le n\right\}.$$
    In fact, given a symmetric pair of a Lie algebra $\g=\k\oplus \p$ corresponding to a Hermitian symmetric space, we can set the roots of $\k$ to be the regular roots and the roots of $\p$ to be the singular roots. Other exceptional examples include
    \begin{itemize}
    \item $\Delta_{reg}=\left\{\pm\eps_i\pm\eps_j,\pm2\eps_i\right\},\quad \Delta_{sing}=\left\{\pm\eps_i\pm\eps_j\pm\eps_k\right\}$, where $1\le i\ne j\ne k\le 4$.
    
    \item $\Delta_{reg}=\left\{\pm\eps_i\pm\eps_j,\pm\eps_i\ :\ 1\le i\ne j\le 5\right\}$,\quad $\Delta_{sing}=\left\{\frac{1}{2}\left(\pm\eps_1\pm\eps_2\pm\eps_3\pm \eps_4\pm\eps_5\right)\right\}$.
    
        \item $\Delta_{reg}=\left\{\pm2\eps_1,\pm2\eps_2,\pm2\eps_3,\pm2\eps_4\right\},\quad \Delta_{sing}=\left\{\pm\eps_1\pm\eps_2\pm\eps_3\pm\eps_4\right\}.$

    \end{itemize}
\end{example}

\begin{remark}
    In \cite[Sec. 6.1]{Sh2}  the notion of RGRS was defined and a statement similar to Proposition \ref{GRR prop} is proved. An RGRS is a GRR system with the following additional property. Write $R_{reg}=R_1\times\ldots\times R_k$ where $R_i$ is an irreducible component. Let $W_i$ be the corresponding Weyl group and $p_i$ be the projection onto $\operatorname{span} R_i$. For an RGRS, we assume that $p_i(R_{sing})/ \{0\}$ is a union of small $W_i$-orbits. We recall that a $W_i$-orbit $X$ is called small if $x-y\in R_i$ for any $x,y\in W$, $x\ne \pm y$.

    By \cite[Prop. 6.6]{Sh2}, the restricted roots $\Delta(\aa,\g)$ form an RGRS. However, not every GRR system is an RGRS. Consider the first system described in Example \ref{GRR Hermitian example} for $n\ge 5$. Here $\operatorname{span}R_1=\{\eps_1+\cdots+\eps_n\}^\perp$. Take $x=p_1(\eps_1+\eps_2)$, $y=p_1(\eps_3+\eps_4)$ which are in the same orbit whereas  $x-y=\eps_1+\eps_2-\eps_3-\eps_4$ is not a projection of the root (for $n=4$, $x=-y$).  We also note that not every RGRS arize from a symmetric pair, for example the last exceptional system in Example \ref{GRR Hermitian example}.
    
\end{remark}
\black 
\section{The Image of the Chavalley Restriction Map}
We begin proving Theorem A. We start by showing that the image of the Chevalley restriction map lies in the space $I(\aa^*)$. This means that the image satisfies two types of invariance conditions: $W$-invariance and a condition related to roots whose root spaces are purely odd.

To show that the image is $W$-invariant, we note that $\aa\subset \p_{\bar 0}$. This means we can first restrict functions from $\p$ to $\p_{\bar 0}$ and then to $\aa$. The restriction from $\p$ to $\p_{\bar 0}$ gives a $\k_{\bar 0}$-invariant function. 
\red By the Chevalley restriction theorem for  $(\g_{\bar 0},\k_{\bar 0})$, we get that the restriction to $\aa$ of functions $S(\p_{\bar 0})^{\k_{\bar 0}}$ is in $S(\aa^*)^W$.
 \black

We now prove that the second type of invariance condition holds. Suppose that the root space $\alpha$ is purely odd but one of its multiples has an even root vector. Then the derivative condition follows from \red  the following proposition. \black 

Let $\mathfrak{p}_{\alpha}:=\left(\left(\mathfrak{g}_{\alpha}\oplus\mathfrak{g}_{-\alpha}\right)\cap\mathfrak{p}\right)\oplus\mathfrak{a}$ and $\k_\alpha=\left(\g_{\alpha}\oplus\g_{-\alpha}\right)\cap\k$. We prove the derivative condition by first restricting to $S(\p_\alpha)^{\k_\alpha}$.

\begin{proposition} \label{odd invariance condition}
For a singular root $\alpha$ of multiplicity $2n$, the projection $f_{0}$ of $f\in S(\mathfrak{p}^{*})^{\mathfrak{k}}$
to $S(\mathfrak{a}^{*})$ satisfies 
\[
\left(D_{h_{\alpha}}\right)^{k}\left(f_{0}\right)\in\left\langle \alpha\right\rangle 
\]
 for $k=1,3,\ldots,2n-1$, that is $\left(D_{h_{\alpha}}\right)^{k}\left(f_{0}\right)$
is zero on $\ker\alpha$. 
\end{proposition}

Let us examine the action of $\k_\a$ on $\p_\a$. By Lemma \ref{skew sym form}, we can take $e_1,\ldots e_{n},e_1',\ldots,e_{n}'$ to be a basis of $\g_\a$ such that $B_\theta(e_i,e_{j}')\in \mathbb C^* \delta_{ij}$, $B_\theta\left(e_i,e_{j}\right)=B_\theta\left(e_i',e_{j}'\right)=0$
and they are normalized such that
\[\left[e_i+\theta e_i,e_{j}'-\theta e_{j}'\right]=\red -[e_i,\theta e_j']+[\theta e_i,e_j']=-p_\aa \left( [e_i,\theta e_j']\right)=\black \delta_{ij}h_{\alpha}.\]
We use the fact that for any $e,e'\in \g_\a$, $[e,e']=0$, for otherwise, $[e,e']$ is an even root vector of $2\a$.
Note that for any $h\in\mathfrak{a}$,
\[\left[e_i+\theta e_i,h\right]=-\alpha(h)\left(e_i-\theta e_i\right).\]
Let 
$\xi_1,\ldots,\xi_{n},\eta_1,\ldots,\eta_{n}\in \p^*_\alpha$
 be such that
\[\xi_i\left(e_{j}-\theta e_{j}\right)=\eta_i\left(e_{j}'-\theta e_{j}^{'}\right)=\delta_{ij},\]
\[\xi_i\left(e_{j}'-\theta e_{j}'\right)=\eta_i\left(e_{j}-\theta e_{j}\right)=0,\]
and 
\[\xi_i\left(\mathfrak{a}\right)=\eta_i\left(\mathfrak{a}\right)=0.\]
Let $f\in S(\mathfrak{p}_{\alpha}^{*})$.
It follows from the above that
\begin{equation}
\label{formula:action of e+theta e}
\left(e_i+\theta e_i\right).f=\alpha\frac{\partial f}{\partial\xi_i}-\eta_iD_{h_{\alpha}}\left(f\right).
\end{equation}

\begin{proof}[Proof of Proposition \ref{odd invariance condition}]
Let $f\in S(\mathfrak{p}^{*})^{\mathfrak{k}}$. Then 
\[
f|_{\p_\alpha}=f_{0}+\sum f_{i_1,\ldots,i_{m}}^{j_1,\ldots,j_{m'}}\xi_{i_1}\cdots\xi_{i_{m}}\eta_{j_1}\cdots\eta_{j_{m'}}
\]
where $f_{0},f_{i_1,\ldots,i_m}^{j_1,\ldots,j_{m'}}\in S(\mathfrak{a}^{*})$
and the sum runs on all subsets $\left\{ i_1,\ldots,i_{m}\right\} ,\left\{ j_1,\ldots,j_{m'}\right\} \subseteq\left\{ 1,\ldots,n\right\}. $
By
$\mathfrak{k}$-invariance, $\left(e_k+\theta e_k\right).\left(f\right)=0$.
Denote $f_k:=f_{1,\ldots,k}^{1,\ldots,k}$, $k\ge 1$ and $f_{0}^{0}:=f_{0}$.
By (\ref{formula:action of e+theta e}), the $S(\mathfrak{a}^{*})$-coefficient
of $\xi_1\cdots\xi_{k-1}\eta_1\ldots\eta_{k}$ in $L_{e_{k}+\theta e_{k}}\left(f\right)$
is 
\begin{equation}
\alpha f_k-D_{\text{\ensuremath{h_{\alpha}}}}f_{k-1}=0. \label{eq}
\end{equation}
for $k\le n$. 

We claim that for $0\le j\le n-1$ , one has 
\[
\left(D_{h_{\alpha}}\right)^{2j+1}f_{0}\in\left\langle \alpha\right\rangle .
\]
In fact, we claim that $\left(D_{h_{\alpha}}\right)^{2j+1}f_{0}$
is a linear combination of terms of the form $\alpha_1^{i}f_{1,\ldots,k}^{1,\ldots,k}$
for $i$ odd and $k\le2j+1$. We prove this claim by induction on
$j$. 
For $j=0$, it follows from (\ref{eq}).
Suppose it holds for $j-1$ . 
Let us show that it holds for $j$. One
has
\begin{align*}
\left(D_{h_{\alpha}}\right)^{2j+1}f_{0} & =\left(D_{h_{\alpha}}\right)^2\left(\left(D_{h_{\alpha}}\right)^{2j-1}f_{0}\right)\\
 & \stackrel{\text{induction}}{=}\left(D_{h_{\alpha}}\right)^2\left(\sum_{i\text{ odd, }k\le2j-1}a_{i,k}\alpha^{i}f_{k}\right)\\
 & =\left(D_{h_{\alpha}}\right)\left(\sum_{i\text{ odd, }k\le2j-1}a_{i,k}\left(\alpha(h_{\alpha})i\alpha^{i-1}f_{k}+\alpha^{i+1}f_{k+1}\right)\right)\\
 & =\sum_{i\text{ odd, }k\le2j-1}a_{i,k}\left(\alpha(h_{\alpha})^2i(i-1)\alpha^{i-2}f_{k}+\alpha(h_{\alpha})\left(2i+1\right)\alpha^{i}f_{k+1}+\alpha^{i+2}f_{k+2}\right).
\end{align*}
Thus $\left(D_{h_{\alpha}}\right)^{2j+1}f_{0}\in\left\langle \alpha\right\rangle $
and the assertion follows.
\end{proof}

\begin{remark}
Note that one can not continue the argument for $j\ge n$ since (\ref{eq}) does not give information about $D_{h_{\alpha}}f_{n}$.
\end{remark}


\begin{example}
Suppose that $n=3$, that is, the multiplicity of $\mathfrak{g}_{\alpha}$
is 6. Denote $D:=D_{h_{\alpha}}$ and $c:=\alpha(h_{\alpha})$. Then
\[Df_{0}=\alpha f_1\]
\[D^2f_{0}=c^2f_1+\alpha^2f_2\]
\[D^{3}f_{0}=\left(c^2+2c\right)\alpha f_2+\alpha^{3}f_3\]
\[D^{4}f_{0}=\left(c^{3}+2c^2\right)f_2+\left(c^2+2c+3\right)\alpha^2f_3+\alpha^{4}f_{4}\]
\[D^{5}f_{0}=\left(3c^{3}+4c_{4}^2+6c\right)\alpha f_3+\left(c^2+2c+3\right)\alpha^{3}f_{4}+\alpha^{5}f_{5}.\]
\end{example}

\section{Injectivity of the Chevalley restriction map}

The injectivity of the Chevalley restriction map follows from the following general lemma.
\begin{lemma}\label{lem-inj} Let $\g$ be a finite-dimensional Lie superalgebra and $V$ be a finite dimensional $\g$-module \red and assume that $\g_{\bar 1} u=V_{\bar 1}$ for some $u\in V_{\bar 0}$. \black
Then the restriction map $\mathbb C[V]^\g\to \mathbb C[V_{\bar 0}]^{\g_{\bar 0}}$ is injective.
\end{lemma}
\red This lemma is a special case of \cite[Prop. 1] {Serge2}. We give a self contained proof.\black
\begin{proof} Note that the set $U$ of all $u\in V_{\bar 0}$ satisfying the assumption of the lemma is \red Zariski \black  open and hence dense in $V_{\bar 0}$.

Choose a basis $\xi_1,\dots,\xi_n$ for $V_{\bar 1}^*$ and
$x_1,\dots x_m\in V_{\bar 0}^*$. We use the identity
  $$\mathbb C[V]=\mathbb C[V_{\bar 0}]\otimes\Lambda(\xi_1,\dots,\xi_n)$$
  and introduce a $\mathbb Z$-grading on $\mathbb C[V]$ by setting 
  $\mathbb C[V_{\bar 0}]$ to have degree zero and all $\xi_i$-s to have degree $1$. 
  For any $v\in V_{\bar 0}$ we define the evaluation map
  $$\operatorname{ev}_v: \mathbb C[V]\to \Lambda(\xi_1,\dots,\xi_n)$$
  in the natural way. Consider the representation map $\rho:\g\to\g\l(V)$.  
  For every $X\in\g_{\bar 1}$, \red $\rho(X)\in\mathfrak{gl}(V)$   can be written in the form $\rho(X)=X^+ +X^-$
  where $X^+V_{\bar 0}=0$ and $X^-V_{\bar 1}=0$. \black
  In particular, in the action on the dual space $V^*$,
$X^+=\sum_{i,j}a_{ij}\xi_i\frac{\partial}{\partial x_j}$
  and $X^-=\sum_{i,j} b_{ij}y_j\frac{\partial}{\partial \xi_i}$ for some 
   $a_{ij},b_{ij}\in\mathbb C$. 
  Note that $X^{-}$ commutes with $\operatorname{ev}_v$
  for any $v\in V_{\bar 0}$ and  
  $\rho(X)(v)=X^{-}(v)=\sum b_{ij}y_j(v)\frac{\partial}{\partial \xi_i}$ is a derivation in $\Lambda(\xi_1,\dots,\xi_n)$.

  Let $f\in \mathbb C[V]^\g$ lie in the kernel of the restriction map and
  $f_k$ be the lowest-degree nonzero term of $f$ in our grading. We have $k>0$.
  To prove the lemma it suffices to show that $f_k=0$.
  For any $X\in\g_{\bar 1}$, we have $X^-(f_k)=0$ and hence 
  $X^{-}(v)(\operatorname{ev}_v f_k)=0$ for any $v\in V_{\bar 0}$.
  We have $X^+(v)=0$ for any $v\in V_{\bar 0}$ and therefore $\rho(X)(v)=X^-(v)$.
  For a given $u\in U$, the set $\{X^{-}(u)\mid X\in\g_{\bar 1}\}$ contains $\frac{\partial}{\partial \xi_i}$ for all $i=1,\dots n$. Therefore $\frac{\partial}{\partial \xi_i}({ev}_u f_k)=0$ for $i=1,\ldots,n$ and $\operatorname{ev}_u f_k=0$.
  By the density of $U$ we get that $f_k=0$.
  \end{proof}

Now we can prove the following proposition:
\begin{proposition} Let $\g=\k\oplus\p$ have Iwasawa decomposition. Then the restriction map $\mathbb C[\p]^\k\to \mathbb C[\mathfrak a]$ is injective.
\end{proposition}
\begin{proof} \red Let us choose a generic $h\in\mathfrak a$. Then $\operatorname{ad}_h$ is non-degenerate
on any restricted root space $\g_{\alpha}$ and hence 
$$\operatorname{ad}_h: \k\cap(\g_\alpha\oplus\g_{-\alpha})\to \p\cap(\g_\alpha\oplus\g_{-\alpha})$$ is an isomorphism. \black Since $\p_{\bar 1}$ is the sum of
$\p_{\bar 1}\cap(\g_\alpha\oplus\g_{-\alpha})$, we obtain
$[h,\k_{\bar 1}]=\p_{\bar 1}$.
By Lemma \ref{lem-inj} the restriction map $\mathbb C[\p]^\k\to \mathbb C[\p_{\bar  0}]^{\k_{\bar{0}}}$ is injective. By classical result  the restriction map
$\mathbb C[\p_{\bar  0}]^{\k_{\bar{0}}}\to\mathbb C[\mathfrak a]$ is injective. The statement follows.
  \end{proof}

\section{Surjectivity of the Chevalley restriction map.}
We describe $I(\aa^*)$ in all cases and show that the Chevalley restriction map $C_\theta$ surjects onto $I(\aa^*)$. 
This completes the proof of Theorem A. 
The computation of the set of restricted roots in each case, as well as of $h_\a$ for every singular root is given in the Appendix. 
We also give the image of the restriction map from $I(\h^*)$ to $I(\aa^*)$ and prove Theorem B.

For most cases, we use the following argument to show that $C_\theta$ is surjective. Take the following commutative diagram of restriction maps,
\[ 
\begin{tikzcd}
S(\g^*)^\g \arrow[r,"C"] \arrow[d, "R_\theta" ]
& I(\h^*) \arrow[d, "R" ] \\
S(\p^*)^\k \arrow[r,  "C_\theta" ]
&  I(\aa^*)
\end{tikzcd}\]
Note that $C$ is an isomorphism by \cite{Serge2} and $C_\theta$ is an embedding. 
\red 
In Theorem A below, we show that $C_\theta$ is always an isomorphism. 
If $R$ is surjective then this is straightforward---for the remaining cases we give a separate argument. As a consequence we deduce that $R$ is surjective if and only if $R_\theta$ is surjective, which gives Theorem B.
\black 
\subsection{Case $\g=\mathfrak{gl}(m|2n), \mathfrak k=\mathfrak{osp}(m|2n)$.}\label{surj_gl_osp}

Here
$$I(\aa^*)=\left\{f\in\mathbb C[\eps_1,\ldots, \eps_m,\delta_1,\ldots,\delta_n]^{S_m\times S_n} \mid \left(\frac{2\partial}{\partial \eps_1}+\frac{\partial}{\partial \delta_1} \right)f\in \left<\eps_1-\delta_1\right> \right\},$$
(see Case \ref{gl_osp} in the Appendix).

By \cite[Thm. 2]{SV}, $I(\aa^*)$ is generated by $\phi_k=\eps_1^k+\ldots+\eps_{m}^{k}-2\delta_1^{k}-\dots-2\delta_{n}^{k}$, $k\in\mathbb Z_{\ge 1}$, which are the projections of $\bar\phi_k=\beps_1^{k}+\ldots+\beps_{m}^{k}-\bdelta_1^{k}-\dots-\bdelta_{2n}^{k}$ from $I(\h^*)$.
Hence the restriction map from $I(\h^*)$ to $I(\aa^*)$ is surjective and so is $R_\theta$.

\red It is natural to consider the same restriction of the involution to of $\mathfrak{sl}(m|2n)$. We let $\tilde\aa=\aa\cap\mathfrak{sl}(m|2n)=\{a\in\aa\mid\bar\phi_1(a)=0\}$.
  
\begin{proposition}
    The ring $S((\p\cap\mathfrak{sl}(m|2n))^*)^\k$ is isomorphic to $I(\tilde\aa^*)$. Moreover 
    $I(\tilde\aa^*)\cong I(\aa^*)/\langle \phi_1\rangle$.
    
\end{proposition}
\begin{proof} 
Let $\tilde\h=\h\cap\mathfrak{sl}(m|2n)$. 
We  show that the restriction map $\tilde R:I(\tilde\h^*)\rightarrow I(\tilde \aa^*)$ is surjective. Consider the following diagram 
\[ 
\begin{tikzcd}
I(\h^*) \arrow[r,"R"] \arrow[d, "S " ]
& I(\aa^*) \arrow[d, " T" ] \\
I(\tilde\h^*)\arrow[r,  "\tilde R " ]\arrow[hookrightarrow,d, " " ]
&  I(\tilde \aa^*)\arrow[hookrightarrow,d, " " ]\\
J(\tilde\h^*)\arrow[r,"R_J"]  &J(\tilde \aa^*)
\end{tikzcd}\]
where the rings $J(\tilde\h^*)$ and $J(\tilde\aa^*)$ are "slightly larger" rings than $I(\tilde\h^*)$ and $I(\tilde\aa^*)$, namely
\[J(\tilde\h^*):=\left\{f\in\mathbb C[\tilde\h^*]^{S_{m-1}\times S_{2n}} \mid \left(\frac{\partial}{\partial \beps_1}+\frac{\partial}{\partial \bdelta_1} \right)f\in \left<\beps_1-\bdelta_1\right> \right\}\]
\[J(\tilde\aa^*):=\left\{f\in\mathbb C[\tilde\aa^*]^{S_{m-1}\times S_n} \mid \left(\frac{2\partial}{\partial \eps_1}+\frac{\partial}{\partial \delta_1} \right)f\in \left<\eps_1-\delta_1\right> \right\}\]
and $R,\tilde R, R_J,S,T$ are restriction maps. 
Here $\beps_i,\bdelta_j$ and $\eps_i,\delta_j$ denote their restriction to $\tilde\h^*$ and $\tilde\aa^*$, respectively (they are linearly dependent).
We have shown that $R$ is surjective, and $S$ is surjective by \cite[Sec 0.6.2]{SV}. Let us show that $\tilde R$ is surjective.

Let $S_m$ be the permutation group of $\{\eps_1=\beps_1,\eps_2=\beps_2,\ldots,\eps_m=\beps_m\}$. Then $I(\tilde\h^*)=J(\tilde\h^*)^{S_m}$ and $I(\tilde\aa^*)=J(\tilde\aa^*)^{S_m}$.
The map $R_J$ is surjective due to the surjectivity of the map $R$ for the case $\g=\mathfrak{gl}(m-1|2n), \mathfrak k=\mathfrak{osp}(m-1|2n)$.
Since $\tilde R, R_J$ are $S_m$-equivariant, the surjectivity of $R_J$ implies surjectivity on $S_m$-invariants (and in fact surjectivity on every isotypic component of the $S_m$-module $J(\tilde\aa^*)$). This precisely give the surjectivity of $\tilde R$.

To show that  $I(\tilde\aa^*)\cong I(\tilde\aa^*)/\langle \phi_1\rangle$, we first note that $T$ is surjective because $\tilde R,S,R$ are surjective. The kernel of $T$ consists of polynomials which are zero on $\a^*\cap \mathfrak sl(m|2n)$. Thus they are divisible by $\phi_1$. 
\end{proof}
\begin{remark}
    When $m\ne 2n$, it is easier to show that $\tilde R$ is surjective. Indeed, given $f\in I(\tilde\aa^*)$, we can take the preimage $\bar f(h)=f(h-\frac{\mathfrak{str}h}{m-2n}I_{m+2n})$.
\end{remark}

When $m=2n$, once can also consider the involution acting on $\mathfrak{psl}(2n|2n)=\mathfrak{sl}(2n|2n)/\mathbb C I_{4n}$. We note that $I_{4n}\in\tilde \aa$ and that $I_{4n}$ is the center of $\g$ and in particular $\k$-invariant. We obtain the following corollary.
\begin{corollary}
    The ring $S(((\p\cap\mathfrak{sl}(2n|2n))/\mathbb C{I_{4n}})^*)^\k$ is isomorphic to 
    $$I(\tilde \aa^*)\cap \mathbb C\left[\eps_i-\eps_j, \eps_i-2\d_k, \d_l-\d_k \mid i,j=1,\ldots,2n;k,l=1,\ldots,n   \right].$$
\end{corollary}
\black

 \black
\subsection{Case $\g=\mathfrak{gl}(2m+a|2n+b)$, $\k=\mathfrak{gl}(m|n) \oplus \mathfrak{gl}(m+a|n+b)$, $a,b \geq 0$.}

Here
\begin{align*}
    I(\aa^*)=\left\{f\in\mathbb C[\eps_1^2,\ldots,\eps_{m}^2,
\delta_1^2,\ldots,\delta_{n}^2]^{S_{m}\times S_{n}} \mid  
\left(\frac{\partial}{\partial \eps_1}+\frac{\partial}{\partial \delta_1} \right)f\in \left<\eps_1-\delta_1\right> \right\}
\end{align*}
(see Case \ref{gl_plus_gl} in the Appendix).

By \cite[Thm. 2]{SV}, $I(\aa^*)$ is generated by $\phi_{2k}=\eps_1^{2k}+\ldots+\eps_{m}^{2k}-\delta_1^{2k}-\dots-\delta_{n}^{2k}$, $k\in\mathbb Z_{\ge 1}$, which are the projections of  $\bar\phi_{2k}=\frac{1}{2}(\beps_1^{2k}+\ldots+\beps_{2m+a}^{2k}-\bdelta_1^{2k}-\dots-\bdelta_{2n+b}^{2k})$ from $I(\h^*)$.
Hence the restriction map from $I(\h^*)$ to $I(\aa^*)$ is surjective and so is $R_\theta$.
\subsection{Case $\g=\mathfrak{osp}(2m+a|4n+2b)$, $\mathfrak k=\mathfrak{osp}(m|2n)\oplus \mathfrak{osp}(m+a|2n+2b)$, $a\ge 1$, $b\ge 0$.}
Here
$$I(\aa^*)=\left\{f\in\mathbb C[\eps_1^2,\ldots,\eps_{m}^2,\delta_1^2\ldots\delta_{n}^2]^{S_{m}\times S_n}\mid  \left(\frac{2\partial}{\partial\eps_1}+\frac{\partial}{\partial\delta_1} \right)f\in \left<\eps_1-\delta_1\right> \right\}$$
(see Case \ref{osp_plus_osp} in the Appendix).
Similarly to the case in Section \ref{surj_gl_osp}, $I(\aa^*)$ is generated by $\phi_{2k}=\eps_1^{2k}+\ldots+\eps_{m}^{2k}-2\delta_1^{2k}-\dots-2\delta_{n}^{2k}$, $k\in\mathbb Z_{\ge 1}$, which are the projections of $\bar\phi_{2k}=\beps_1^{2k}+\ldots+\beps_{2m+a}^{2k}-\bdelta_1^{2k}-\dots-\bdelta_{2n+b}^{2k}$ in $I(\h^*)$.
Hence the restriction map from $I(\h^*)$ to $I(\aa^*)$ is surjective and so is $R_\theta$.

\subsection{Case $\g=\mathfrak{osp}(2m|4n+2b), \mathfrak k=\mathfrak{osp}(m|2n)\oplus \mathfrak{osp}(m|2n+2b)$, $b\ge 0$.}
\label{sec:hard D-Case}
This case is also part of Case \ref{osp_plus_osp}, described in the Appendix. However, here the root $\eps_i$ is a singular root.
The Weyl group $W$ acts on $\eps_1,\ldots,\eps_m$ as a group of type $D$ and we have the additional condition that $\left(\frac{\partial}{\partial\eps_1}\right)^k f\in \left<\eps_1\right>$ for $k=1,3\ldots,2b-1$. Thus
$I(\aa^*)= I_1\left(\mathfrak{a}^{*}\right)\oplus I_2 \left(\mathfrak{a}^{*}\right)$ where

\begin{align*}
I_1\left(\mathfrak{a}^{*}\right)=
 & \left\{ f\in\mathbb{C}\left[\eps_1^2,\ldots,\eps_m^2,\d_1^2,\ldots,\d_n^2\right]^{S_m\times S_n}\mid \left(\frac{2\partial}{\partial \eps_1}+\frac{\partial}{\partial \d_1}\right)f\in\left\langle \eps_1-\d_1\right\rangle \right\} ,\\
I_2 \left(\mathfrak{a}^{*}\right)= 
& \left\{ f\in\left(\eps_1\cdots \eps_m\right)^{2b+1}\mathbb{C}\left[\eps_1^2,\ldots,\eps_m^2,\d_1^2,\ldots,\d_n^2\right]^{S_{m}\times S_n}\mid \left(\frac{2\partial}{\partial \eps_1}+\frac{\partial}{\partial \d_1}\right)f\in\left\langle \eps_1-\delta_1\right\rangle \right\} .
\end{align*}

We prove surjectivity  by showing that the dimensions of
$S(\mathfrak{p}^{*})_{d}^{\mathfrak{k}}$ and $I(\mathfrak{a}^{*})_{d}$
are equal for any degree $d$. We express these dimensions in term
of certain partitions: a partitions $\lambda$ of $d$ is called \emph{regular}
if it is contained in the fat $\left(m|2n \right)$-hook and either
\begin{itemize}
\item All the parts $\lambda_i$ are even; or
\item $\lambda_1,\ldots,\lambda_m$ are odd, $\lambda_m\ge2n+2b+1$
and $\lambda_i$ is even for $i\ge m+1$.
\end{itemize}
In the first case, we call $\lambda$ an \emph{even regular} partition, and in the second case, an \emph{odd regular} partition.
Note that the diagram of an odd regular partition must contain the $m\times (2n+2b+1)$ rectangle.
\begin{example}
Let $m=2$, $n_1=2$ and $b=1$. Then the following paritions
are regular:
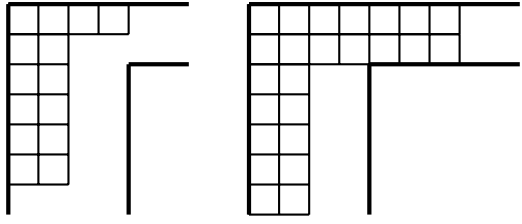
\usetikzlibrary{math}  
$$
\begin{tikzpicture}[scale=0.40]
\tikzmath{\m=1; \n=3; \tail = 2; \shift =8;  \m1=2; \n1=2; }
\tikzmath{\i1 = 4; \i2=7; \i3=3; \j1 = 5; \j2=4;\j3=0;}
\draw[step=1cm,black,thick] (0,0) grid (\i1,-1);
\draw[step=1cm,black,thick] (0,-\m1+1) grid (2,-\m1+1-\j1);
\draw [ultra thick] (0,0) -- (2*\n1+\tail,0);
\draw [ultra thick] (0,0) -- (0,-\m1-5);
\draw [ultra thick] (2*\n1,-\m1) -- (2*\n1,-\m1-5);
\draw [ultra thick] (2*\n1,-\m1) -- (2*\n1+\tail,-\m1);
\tikzmath{\i1 = 7; \i2=7; \i3=3; \j1 = 5; \j2=4;\j3=0;}
\draw[step=1cm,black,thick] (0+\shift,0) grid (\i1+\shift,-1);
\draw[step=1cm,black,thick] (0+\shift,-1) grid (\i2+\shift,-2);
\draw[step=1cm,black,thick] (0+\shift,-\m1) grid (2+\shift,-\m1-\j1);
\draw [ultra thick] (0+\shift,0) -- (\i1+\tail+\shift,0);
\draw [ultra thick] (0+\shift,0) -- (0+\shift,-\m1-5);
\draw [ultra thick] (2*\n1+\shift,-\m1) -- (2*\n1+\shift,-\m1-5);
\draw [ultra thick] (2*\n1+\shift,-\m1) -- (\i1+\tail+\shift,-\m1);
\end{tikzpicture}$$
The following partition is not regular (but is regular for $b =0$):
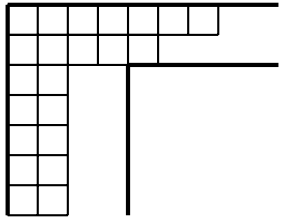
\usetikzlibrary{math} 
$$
\begin{tikzpicture}[scale=0.40]
\tikzmath{\m=1; \n=3;
\tail = 2; \shift =0;  \m1=2; \n1=2; }
\tikzmath{\i1 = 7; \i2=5; \i3=3; \j1 = 5; \j2=4;\j3=0;}
\draw[step=1cm,black,thick] (0+\shift,0) grid (\i1+\shift,-1);
\draw[step=1cm,black,thick] (0+\shift,-1) grid (\i2+\shift,-2);
\draw[step=1cm,black,thick] (0+\shift,-\m1) grid (2+\shift,-\m1-\j1);
\draw [ultra thick] (0+\shift,0) -- (\i1+\tail+\shift,0);
\draw [ultra thick] (0+\shift,0) -- (0+\shift,-\m1-5);
\draw [ultra thick] (2*\n1+\shift,-\m1) -- (2*\n1+\shift,-\m1-5);
\draw [ultra thick] (2*\n1+\shift,-\m1) -- (\i1+\tail+\shift,-\m1);
\end{tikzpicture}$$
\end{example}

\begin{lemma}
The dimension of $S(\mathfrak{p}^*)_d^{\mathfrak k}$ is at least the numbers of regular partitions whose parts sum to $d$.
\end{lemma}
\begin{proof}
In this case $\mathfrak{p^{*}}$ is isomorphic to $\mathbb{C}^{m|2n }\otimes\mathbb{C}^{m|2n+2b }$
as $\mathfrak{k}$-modules. We view $\mathbb{C}^{m|2n }\otimes\mathbb{C}^{m|2n+2b }$
as a $\mathfrak{g}':=\mathfrak{gl}(m|2n )\oplus\mathfrak{gl}(m|2n+2b )$-module
which was restricted to $\mathfrak{k}$. By \cite[Thm. 3.2]{CW},
we have the \red  following \black  decomposition of $\mathfrak{g}'$-modules
\[
S \left(\mathbb C^{m|2n}\otimes\mathbb{C}^{m|2n+2b}\right)_d\cong\bigoplus_{\lambda}V_{m|2n}^{\lambda}\otimes V_{m|2n+2b}^{\lambda},
\]
where the sum is over all paritions $\lambda$ of $d$ satisfying
$\lambda_{m+1}\le 2n \le 2n+2b $. Here the highest weight of the
simple modules $V_{m|2n}^{\lambda}$ and $V_{m|2n+2b}^{\lambda}$ is 
\[
\lambda=\lambda_1 \varepsilon_1 +\ldots\lambda_{m}\varepsilon_{m}+\left\langle \lambda'_1 -m\right\rangle \d_1 +\ldots+\left\langle \lambda'_{2n }-m\right\rangle \d_{2n },
\]
 where $\left\langle r\right\rangle :=\max\left\{ 0,r\right\} $ (for
$V_{m|2n+2b }^{\lambda}$ there is zero multiple of $\delta_{2n +1}+\ldots+\delta_{2n+2b }$).

Suppose that $\lambda$ is an even regular partition. 
 
By \cite[Prop. 4.6]{SS}\footnote{In the notation of \cite{SS}, the partitions do not have to be even but the correspondence between partions to weights is ``doubled'',
see \cite[(32)]{SS}}, 
the dimension of the invariants under  $\mathfrak{osp}\left(m|2n\right)$ (resp. $\mathfrak{osp}\left(m|2n+b\right)$) in 
$V_{m|2n}^{\lambda}$ (resp. $V_{m|2n+2b}^{\lambda}$) is one.
Hence $V_{m|2n }^{\lambda}\otimes V_{m|2n+b }^{\lambda}$ contains
a nonzero $\mathfrak{k}$-fixed vector.

We are left to construct a nonzero $\mathfrak{k}$-fixed vector for
every odd regular partition $\mu$. Let $\mathfrak{str}_{m|2n}$ and $\mathfrak{str}_{m|2n+2b}$ 
be the super-trace module over $\mathfrak{gl}\left(m|2n\right)$ and $\mathfrak{gl}\left(m|2n+2b\right)$, respectively.
Then 
\begin{align*}
    V_{m|2n}^{\lambda}\otimes\mathfrak{str}_{m|2n}\cong V_{m|2n}^{\mu},\quad
    V_{m|2n+2b}^{\lambda}\otimes\mathfrak{str}_{m|2n+2b}\cong V_{m|2n+2b}^{\mu}.
\end{align*}
where $\lambda$ is the partition obtained from $\mu$ by removing
a box from the first $n$-rows and adding a box to the first $2n$ (resp. $2n+2b$)
columns. Note that this can be done since $\lambda_{m}\ge2n+2b +1\ge2n +1$). 
Moreover,
tensoring with the super-trace module yields an isomorphic module over $\mathfrak{osp}\left(m|2n\right)$ (resp. $\mathfrak{osp}\left(m|2n+2b\right)$)
and $\lambda$ is an even regular partition. Hence $$\dim\left(V_{m|2n}^{\mu}\right)^{\mathfrak{osp}\left(m|2n\right)}=\dim\left(V_{m|2n+2b}^{\mu}\right)^{\mathfrak{osp}\left(m|2n+2b\right)}=1,$$
and $V_{m|2n}^{\mu}\otimes V_{m|2n+2b}^{\mu}$ contains a $\mathfrak{k}$-invariant
vector as desired.
\end{proof}

\begin{lemma}
\label{lem:upper bound on dimension}
The dimension of $I(\aa^*)_d$ is at most the \red  number \black 
of regular partitions whose parts sum to $d$.
\end{lemma}

To prove the lemma, we show that the leading term of an $f\in I(\aa^*)_d$ corresponds to a regular partition. 
However, the order on monomials and the correspondence to partitions is not the standard one. We demonstrate the idea of the proof in the following example.
\begin{example}
\label{exa:monomials to partitions}Let $m=3$, $n=3$ and $b=0$.
We take the order on monomials corresponding to the lexical order
on $x_1>x_2>y_1>y_2>y_3>x_2$. The monomial $x_1^{11}x_2^{11}y_1^{10}y_2^8 y_3^4 x_3^3$
corresponds to the following regular partition
$$
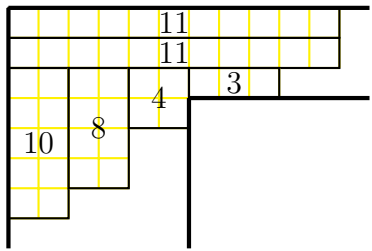
\begin{tikzpicture}[scale=0.40]

\draw (3+3,-3-6) node {$(11,11|10,8,4|3)$};

\draw[step=1cm,yellow,thick] (0,0) grid (11,-1);
\draw[step=1cm,black,thick] (0,0) rectangle  node {$11$} (11,-1);

\draw[step=1cm,yellow,thick] (0,-1) grid (11,-2);
\draw[step=1cm,black,thick] (0,-1) rectangle  node {$11$} (11,-2);

\draw[step=1cm,yellow,thick] (2*3,-3+1) grid (2*3+3,-3);
\draw[step=1cm,black,thick] (2*3,-3+1) rectangle  node {$3$} (2*3+3,-3);

\draw[step=1cm,yellow,thick] (0,-3+1) grid (2,-3+1-5);
\draw[step=1cm,black,thick] (0,-3+1) rectangle  node {$10$} (2,-3+1-5);

\draw[step=1cm,yellow,thick] (2,-3+1) grid (4,-3+1-4);
\draw[step=1cm,black,thick] (2,-3+1) rectangle  node {$8$} (4,-3+1-4);

\draw[step=1cm,yellow,thick] (4,-3+1) grid (6,-3+1-2);
\draw[step=1cm,black,thick] (4,-3+1) rectangle  node {$4$} (6,-3+1-2);
\draw [ultra thick] (0,0) -- (2*3+6,0);
\draw [ultra thick] (0,0) -- (0,-3-5);
\draw [ultra thick] (2*3,-3) -- (2*3,-3-5);
\draw [ultra thick] (2*3,-3) -- (2*3+6,-3);
\end{tikzpicture}
$$
The monomial $x_1 ^{11}x_2 ^{7}y_1 ^{10}y_2 ^{8}y_{3}^{0}x_{3}^{3}$
supposedly corresponds to the following shape but we show that it
can not be a leading monomial of an element in $I(\mathfrak{a}^{*})$.
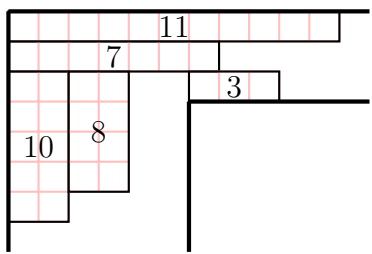
$$
\begin{tikzpicture}[scale=0.40]

\draw[step=1cm,pink,thick] (0,0) grid (11,-1);
\draw[step=1cm,black,thick] (0,0) rectangle  node {$11$} (11,-1);

\draw[step=1cm,pink,thick] (0,-1) grid (7,-2);
\draw[step=1cm,black,thick] (0,-1) rectangle  node {$7$} (7,-2);

\draw[step=1cm,pink,thick] (2*3,-3+1) grid (2*3+3,-3);
\draw[step=1cm,black,thick] (2*3,-3+1) rectangle  node {$3$} (2*3+3,-3);

\draw[step=1cm,pink,thick] (0,-3+1) grid (2,-3+1-5);
\draw[step=1cm,black,thick] (0,-3+1) rectangle  node {$10$} (2,-3+1-5);

\draw[step=1cm,pink,thick] (2,-3+1) grid (4,-3+1-4);
\draw[step=1cm,black,thick] (2,-3+1) rectangle  node {$8$} (4,-3+1-4);

\draw [ultra thick] (0,0) -- (2*3+6,0);
\draw [ultra thick] (0,0) -- (0,-3-5);
\draw [ultra thick] (2*3,-3) -- (2*3,-3-5);
\draw [ultra thick] (2*3,-3) -- (2*3+6,-3);
\draw (3+3,-3-6) node {$(11,7|10,8,0|3)$};
\end{tikzpicture}$$ 
Our proof shows that in this case, the powers of $y_1 ,y_2 ,y_{3}$ are greater or equal to $2$ and that the power of $x_2$ is at
least the power of $x_3$ plus 6. 
\end{example}

\begin{proof}[Proof of Lemma \ref{lem:upper bound on dimension}]
By \cite[Prop. 2]{SV}, the number of even regular partitions of
total size $d$ is equal to the dimension of $I_1 \left(\aa^{*}\right)_{d}$.
We are left to show that the dimension of $I_2 (\aa^{*})_{d}$
is equal to the number of odd regular partitions of total size $d$. 

Let $f\in I_2 (\mathfrak{a}^{*})_{d}$ and denote by 
\[
\left(\lambda_1 ,\ldots,\lambda_{m-1}\ |\ \mu_1 ,\ldots,\mu_n\ |\ \lambda_m\right):=
x_1 ^{\lambda_1 }\cdots x_{m-1}^{\lambda_{m-1}}y_1 ^{\mu_1 }\cdots y_{n }^{\mu_n}x_{m}^{\lambda_m}
\]
 the leading monomial of $f$ with respect to the order that corresponds
to the lexical order 
\[
x_1 >\ldots>x_{m-1}>y_1 >\ldots>y_{n }>x_{m}.
\]
First we claim that $\mu_{i}\ge2$ for all $i$. Suppose that $\mu_{j}=0$
for some $j$. Since $\left(\frac{2\partial}{\partial x_{m}}+\frac{\partial}{\partial y_{j}}\right)f\mid_{x_{m}=y_{j}}=0$,
there should be another monomial in $f$ of the form 
\[
\left(\lambda_1 ,\ldots,\lambda_{m-1}\ |\ \mu_1 ,\ldots,r,\ldots,\mu_{n }\ |\ \lambda_{m}-r\right),
\]
 where $r\in2\mathbb{Z}_{\ge2}$ is the power of $x_{j}$. This contradicts
the maximality of the first monomial.

Second, we claim that $\lambda_{m-1}\ge\lambda_{m}+2n $. Indeed,
since $\left(\frac{2\partial}{\partial x_{m}}+\frac{\partial}{\partial y_{n }}\right)f\mid_{x_{m}=y_{n }}=0$
and by maximality of $\left(\lambda_1 ,\ldots,\lambda_{m-1} | \mu_1 ,\ldots,\mu_{n }|\lambda_{m}\right)$, we
get that 
\[
\left(\lambda_1 ,\ldots,\lambda_{m-1}\ |\ \mu_1 ,\ldots,\mu_{n }-2r_0 \ |\ \lambda_{m}+2r_0 \right)
\]
is also a monomial in $f$ for some $r_0 \ge1$. 
Next note that since
$\left(\frac{2\partial}{\partial x_{m}}+\frac{\partial}{\partial y_{n -1}}\right)f\mid_{x_m=y_{n-1}}=0$
and by maximality of $\left(\lambda_1 ,\ldots,\lambda_{m-1} | \mu_1 ,\ldots,\mu_n | \lambda_{m}\right),$we
get that 
\[
\left(\lambda_1 ,\ldots,\lambda_{m-1}\ |\ \mu_1 ,\ldots,\mu_{n-1}-2r_1,\mu_n -2r_0\  |\ \lambda_m+2r_0 +2r_1 \right)
\]
is also a monomial of $f$ for some $r_1 \ge1$. Repeating this argument
with $\left(\frac{2\partial}{\partial x_{m}}+\frac{\partial}{\partial y_i}\right)$ for $i=n_1 -3,\ldots,1$
we get that
\[
\left(\lambda_1 ,\ldots,\lambda_{m-1}\ |\ \mu_1 -2r_{n-1},\ldots,\mu_{n }-2r_0 \ |\ \lambda_{m}+2\sum_{i=0}^{n-1}r_{i}\right)
\]
is a monomial of $f$ for $r_0 ,\ldots,r_{n-1}\ge1$. Suppose by contradiction that $\lambda_{m-1}\le2n$. 
By the $S_{m}$-symmetry, 
\[
\left(\lambda_1 ,\ldots,\lambda_{m}+2\sum_{i=0}^{n-1}r_{i}\ |\ \mu_1 -2r_{n-1},\ldots,\mu_{n}-2r_0 \ |\ \lambda_{m-1}\right)
\]
is also a monomial of $f$ which is bigger than the leading monomial.
Thus, $\lambda_{m-1}\ge\lambda_{m}+2\sum_{i=0}^{n-1 }r_{i}\ge\lambda_{m}+2n $. 

Now, given that $\left(\lambda_1 ,\ldots,\lambda_{m-1}|\mu_1 ,\ldots,\mu_{n}|\lambda_{m}\right)$
is such that $\mu_1 ,\ldots,\mu_n\ge2$ and $\lambda_{m-1}\ge\lambda_{m}+2n$,
we can associate the following partition to $f$: the first $m$ rows
are given by $\left(\lambda_1 ,\ldots,\lambda_{m-1},\lambda_{m}+2n\right)$
and the following rows are given by the transpose of the partition
$\left(\frac{\mu_1 }{2}-1,\frac{\mu_1 }{2}-1,\ldots,\frac{\mu_n}{2}-1,\frac{\mu_n}{2}-1\right)$
(see Example \ref{exa:monomials to partitions}). 
\end{proof}

By \cite[0.6.5]{Serge2}, the elements of $I(\h^*)$ are either such that all variables have an even degree or divisible by $\prod_{1\le i\le 2m,1\le j\le 4m+2b} (\beps_i-\bdelta_j)$. Hence the image of the restriction from $I(\h^*)$ to $I(\aa^*)$ is generated be the elements.
\begin{align*}
    &\eps_1^{2k}+\ldots+\eps_{2m}^{2k}-\delta_1^{2k}-\ldots-\delta_{n}^{2k},\\
    &(\eps_1\cdots\eps_m)^{2b+1}
    \prod(\eps_i^2-\d_j^2)^2
    (\eps_1^{2k}+\ldots+\eps_{2m}^{2k})(\delta_1^{2l}+\ldots+\delta_{n}^{2l}),\quad k,l\in\mathbb Z_{\ge 0}.
\end{align*}
We note that the minimal-degree element in $I(\aa^*)$ for which the degree $\eps_1$ is odd has degree $m(2n+2b+1)$. However the minimal-degree element in the image of $I(\h^*)$ for which the degree $\eps_1$ is odd has degree $m(4n+2b+1)$. Hence the restriction map from $I(\h^*)$ to $I(\aa^*)$ is not surjective and $R_\theta$ is not surjective either.

\subsection{Case $\g=\mathfrak{osp}(2m|2n) , \mathfrak k=\mathfrak{gl}(m|n)$.}
Let $m_1:=\left\lfloor \frac{m}{2}\right\rfloor$, then
$$I(\aa^*)=\left\{f\in\mathbb C[\eps_1^2,\ldots,\eps_{m_1}^2,\delta_1^2\ldots\delta_n^2]^{S_{m_1}\times S_n}\mid  \left(\frac{\partial}{\partial\eps_1}+\frac{2\partial}{\partial\delta_1} \right)f\in \left<\eps_1-\delta_1\right> \right\}$$
(see Case \ref{osp_gl_odd} and Case \ref{osp_gl_even} in the Appendix).

Similarly to the case in Section \ref{surj_gl_osp}, $I(\aa^*)$ is generated by $\phi_{2k}=2\eps_1^{2k}+\ldots+2\eps_{m_1}^{2k}-\delta_1^{2k}-\dots-\delta_n^{2k}$ which are the projections of $\bar\phi_{2k}=\beps_1^{2k}+\ldots+\beps_{m}^{2k}-\bdelta_1^{2k}-\dots-\bdelta_n^{2k}$ in $I(\h^*)$.
Hence the restriction map from $I(\h^*)$ to $I(\aa^*)$ is surjective and so is $R_\theta$.

\subsection{Case $\g=F_4 , \mathfrak k=\mathfrak{gosp}(2|4)$.}

Here
$$I(\aa^*)=\left\{f\in\mathbb C[\eps_1^2,\eps_2^2,\d^2]^{S_2} \mid  \left(\frac{\partial}{\partial \eps_1}+\frac{3}{2}\frac{\partial}{\partial \d} \right)f\in \left<\eps_1-\d\right> \right\},$$
 (see Case \ref{f4_gosp24} in the Appendix).
By \cite[Prop. 2]{SV}, $\dim I(\aa^*)_{2d}$ is at most the number of partitions of $d$ which fit into the $(1|2)$-fat hook. 

Let us show that $\dim S(\p^*)^\k_{2d}$ is at least the number of such partitions.
The $\k$-module $\p^*$ is isomorphic to $V\oplus V^*$ where $V$ is the module described in \cite[Sec. 3.2]{SSS}. Hence $S(\p^*)\cong S(V)\otimes S(V)^*$. By \cite[Prop. 3.6]{SSS},  $S(V)_d=\oplus_{\lambda\in E_d} L(\lambda)$ where $E_d$ is a set of weights in bijection with $\{(q,2r,s)\mid q+2r+s=d;\ q,r\ge 0, s\ge 2 \}\cup \{d\}$. This set is in bijection with the set of partitions of $d$ in the $(1|2)$-fat hook by 
$$
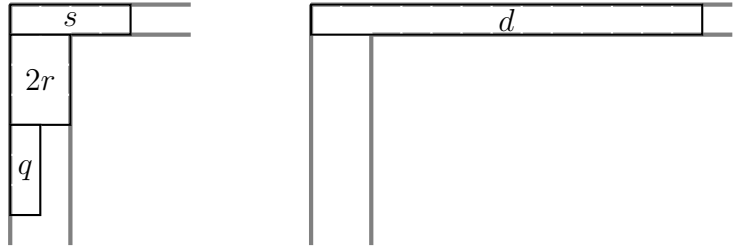
\begin{tikzpicture}[scale=0.40]

\draw [gray,ultra thick] (0,0) -- (6,0);
\draw [gray,ultra thick] (0,0) -- (0,-8);
\draw [gray,ultra thick] (2,-1) -- (2,-8);
\draw [gray,ultra thick] (2,-1) -- (6,-1);

\draw[step=1cm,white,thick] (0,0) grid (4,-1);
\draw[step=1cm,black,thick] (0,0) rectangle  node {$s$} (4,-1);

\draw[step=1cm,white,thick] (0,-1) grid (2,-3-1);
\draw[step=1cm,black,thick] (0,-1) rectangle  node {$2r$} (2,-3-1);

\draw[step=1cm,white,thick] (0,-4) grid (1,-7);
\draw[step=1cm,black,thick] (0,-4) rectangle  node {$q$} (1,-7);

\draw [gray,ultra thick] (0+10,0) -- (14+10,0);
\draw [gray,ultra thick] (0+10,0) -- (0+10,-8);
\draw [gray,ultra thick] (2+10,-1) -- (2+10,-8);
\draw [gray,ultra thick] (2+10,-1) -- (14+10,-1);

\draw[step=1cm,white,thick] (0+10,0) grid (13+10,-1);
\draw[step=1cm,black,thick] (0+10,0) rectangle  node {$d$} (13+10,-1);

\end{tikzpicture}
$$
Applying Schur's Lemma to $S(V)_d\otimes S(V)_d^*\subseteq S(\p^*)_{2d}$, we obtain that the number of linearly independent invariant vectors is at least the number irreducible summands of $S(V)_d$, that is, the number of partitions of $d$ which fit into the $(1|2)$-fat hook. 

By \cite[Thm. 2]{SV}, $I(\aa^*)$ is generated by $\phi_{2k}=\eps_1^{2k}+\eps_2^{2k}-\frac{2}{3}\d^{2k}$, $k\in\mathbb Z_{\ge 0}$.  We note that $\dim I(\aa^*)_4=2$ whereas it follows from \cite[0.6.8]{Serge2} that $\dim I(\h^*)_4=1$ and so the restriction map from $I(\h^*)$ to $I(\aa^*)$ can not be surjective. Thus $R_\theta$ is not surjective in this case.

The image of the restriction from $I(\h^*)$ to $I(\aa^*)$ is generated by the images of the generators of $I(\h^*)$ (which are also listed in the next section), namely
$$\phi_2,\quad 4\phi_6-15\phi_4\phi_2 \ \text{  and } \  (\d^2-\eps_1^2)^2(\d^2-\eps_2^2)^2(\eps_1^{2i}+\eps_2^{2i})\d^{2j},  i,j\in\mathbb Z_{\ge 0}.$$
\subsection{Case $\g=F_4 , \mathfrak k=\mathfrak{sl}(1|4)$.}
Here
\[I(\mathfrak{a}^{*})=
\left\{ f\in\mathbb{C}\left[\eps^2,\d^2\right]\ |
\ \left(\frac{\partial}{\partial\eps}+\frac{3\partial}{\partial\d}\right)^i f\in\left\langle \eps-\d\right\rangle,i=1,3 \right\} \]
(see Case \ref{f4_sl14} in the Appendix).

Let us show that $I(\h^*)$ surjects onto $I(\aa^*)$ and thus $R_\theta$ is also surjective. Similarly to the case in Section \ref{surjectivity case g=D(2,1,a)}, $\dim \mathbb{C}\left[\eps^2,\d^2\right]_{2d}=d+1$. 
Then the condition $D_{h_\alpha}f\in\langle\alpha\rangle$ imposes one linear relation. The condition $D_{h_\alpha}^3 f\in\langle\alpha\rangle$ imposes another linear relation for $d\ge 4$. 
We get that $\dim I(\aa^*)_2 =1$ and $\dim I(\aa^*)_{2d}=d-1$ for $d\ge 2$. 
We claim that $\dim I(\h^*)_2 \mid_\aa\ge 1$ and $\dim I(\h^*)_{2d}\mid_\aa\ge d-1$ for $d\ge 2$.

By \cite[0.6.8]{Serge2}, every element in $I(\h^*)$ is of the form $f=f_0+\prod(\bdelta\pm\beps_1\pm\beps_2\pm\beps_3)\cdot f_1$ where $f_0\in\mathbb C[L_2,L_6]$ and $f_1\in\mathbb C[\h^*]^W$, and 
\begin{align*}
L_2 &:=3(\beps_1^2+\beps_2^2+\beps_3^2)-\bar\d^2, \\
L_6 &:=\bar\d^6+\beps_1^6+\beps_2^6+\beps_3^6\\
& +(\beps_1 - \beps_2)^6+(\beps_2 - \beps_3)^6+(\beps_1 - \beps_3)^6+
(\beps_1 + \beps_2)^6+(\beps_2 + \beps_3)^6+(\beps_1 + \beps_3)^6\\
&-\frac{1}{64}\sum(\bar\d\pm\beps_1\pm\beps_2\pm\beps_3)^6.
\end{align*}

The restriction from $\h^*$ to $\aa^*$ is given by $\eps=\beps_1\mid_\aa$, $\d=\bdelta\mid_\aa$ and $\beps_2\mid_\aa=\beps_3\mid_\aa=0$. For $d\le 3$, we see that the restrictions of $L_2\in I(\h^*)_2$, $L_2^2\in I(\h^*)_4$ and $L_2^3,L_6\in I(\h^*)_6$ to $\aa$ are linearly independent and give the correct dimensions. 
We continue by induction on $d$, that is, assume that $\dim I(\h^*)_{2(d-1)}\mid_\aa\ge d-2$. Then $\dim\left(\left(L_2\cdot I(\h^*)_{2(d-1)}\right)\mid_\aa\right) \ge d-2$. 
Take $f_1\in\mathbb C[\h^*]_{2d-8}^W$ such that $f_1\mid_\aa$ is not divisible by $L_2\mid_\aa$.
Then $f=\prod(\bdelta\pm\beps_1\pm\beps_2\pm\beps_3)\cdot f_1$ is in $ I(\h)^*_{2d}\mid_\aa$ but not in $\left(L_2\cdot I(\h^*)_{2(d-1)}\right)\mid_\aa$. 
Thus, $\dim I(\h^*)_{2d}\mid_\aa\ge d-1$. 

The generators of $I(\aa^*)$ are the restrictions of the generators of $I(\h^*)$, namely of
$$3\eps^2-\d^2,\quad 3\eps^6-\d^6-16(\d^2-\eps^2)^3,\quad (\d^2-\eps^2)^4 \eps^{2i}\d^{2j}, i,j\in\mathbb Z_{\ge 0}.$$
\subsection{Case $\g=D(2,1,a) , \mathfrak k=\mathfrak{osp}(2|2)\oplus \mathfrak{so}(2)$.}\label{surjectivity case g=D(2,1,a)}
Here
$$I(\aa^*)=\left\{f\in\mathbb C[\eps^2,\d^2] \mid  \left((a+1)\frac{\partial}{\partial\eps}+\frac{\partial}{\partial\d} \right)f\in \left<\eps-\d\right> \right\},$$
(see Case \ref{d21a_osp22} in the Appendix).

To show surjectivity in this case, we first note that $\dim I(\aa^*)_{2d}=d$. Indeed, $\dim \mathbb C[\eps^2,\d^2]_{2d}=d+1$ with basis $ \eps^{2d},\eps^{2d-2}\d^2,\ldots,\d^{2d}$. The derivative condition is equivalent to $$ \left. \left((a+1)\frac{\partial}{\partial\eps}+\frac{\partial}{\partial\d} \right)f\right |_{\eps=\d}=0.$$ This gives one linear condition between the basis elements. 

Let us show that $\dim S(\p)^\k_{2d}\ge d$. The $\k$-module $\p^*$ is isomorphic to $V\oplus V^*$ where $S(V)_d$ is the described in \cite[Prop. 3.6]{SSS}: for $a\ne 0,-1$ such that $-\frac{1}{a}\notin\mathbb Q_\le 0$, $S(V)_d$ is a direct sum of $d$ irreducible modules. Hence $S(\p^*)\cong S(V)\otimes S(V)^*$. 
By Schur's lemma, the number of linearly independent invariant vectors is  the number of irreducible summands of  $S(V)_d$ which is $d$, as required.  Since $\dim S(\p)^\k_d=d$ for dense set of $a\in\mathbb C$, it follows that for an arbitrary $a$,  $\dim S(\p)^\k_d\ge d$.

\begin{proposition}
    The ring  $I(\aa^*)$ is generated by the deformed Newton sums $\phi_{2k}=\eps^{2k}-(a+1)\d^{2k}$, $k\ge 1$.
\end{proposition}

\begin{proof}
    
For $a\notin \mathbb Q_{\le -1}$, this follows from \cite[Thm. 2]{SV}. 
We prove the proposition for general $a\ne -1$. 
Let $W_d$ be the subspace of $I(\aa^*)_{2d}$ generated by the deformed Newton sums.
We show that $\dim W_d\ge \dim I(\aa^*)_{2d}$ by induction on $d$. 
For $d=1$, the claim holds since $I(\aa^*)_{2}=\operatorname{span}\{\eps^2-(a+1)\d^2\}$. 
Suppose that $\dim I(\aa^*)_{2d-2}=d-1=\dim I(\aa^*)_{2d-2}$. 
Let $V_d:=(\eps^2-(a+1)\d^2)I(\aa^*)_{2d-2}\subseteq W_d$.
Then $\dim V_d=d-1$.
We show that $\dim W_d\ge d$ by showing that $\eps^{2d}-(a+1)\d^{2d}$ is in $W_d$ but not in $V_d$. 
Suppose otherwise, then $$\eps^{2d}-(a+1)\d^{2d}=(\eps^2-\d^2)(a_0\eps^{2d-2}+a_1\eps^{2d-4}\d^2+\ldots+a_d\d^{2d-2})$$
for some $a_0,\ldots, a_d\in\mathbb C$ such that $a_0\eps^{2d-2}+a_1\eps^{2d-4}\d^2+\ldots+a_d\d^{2d-2}\in I(\aa^*)_{2d-2}$. This equality is impossible unless $1=a_0=a_1=\ldots=a_d=-a-1$. Hence, for $a\ne 2$, the subspace of $I(\aa^*)_{2d}$ generated by the deformed Newton sums is at least $d$ and so the subspace is equal to $d$. For $a=2$, the proposition is known by \cite[Thm. 2]{SV}.

\end{proof}

Note that by \cite[0.6.6]{Serge2}, $\dim I(\h^*)_6=2$. 
Since $\dim I(\aa^*)_6=3$, $I(\h^*)$ does not surject onto $I(\aa^*)$ in this case. 
Hence $R_\theta$ is not surjective in this case.
The image of the restriction from $\dim I(\h^*)$ to $I(\aa^*)$ in this case is generated by $\phi_2$ and elements of the form $(\eps^2-\d^2)^2 \eps^{2i}\d^{2j}$, $i,j\in \mathbb Z$.

\subsection{Case $\g=G_3 , \mathfrak k=D(2,1,3)$.}
Here
$$I(\aa^*)=\mathbb C[\eps^2_1,\eps^2_2,\eps^2_3]^{S_3}/ \left< \eps_1 \eps_2 \eps_3=1\right>$$
(see Case \ref{g3_d21a} in the Appendix).

We claim that in this case $I(\h^*)$ surjects onto $I(\aa^*)$. Indeed, $I(\aa^*)$ is equal to the set of invariant polynomials under the Weyl group of the Lie algebra $G_2$. This set is generated by two algebraically independent elements of degree $2$ and $6$, (see for example \cite[3.7]{H}). Since $\eps_1^2+\eps_2^2+\eps_3^2,\ \eps_1^2\eps_2^2\eps_3^2 $, are algebraically independent and have the suitable degrees, they generate $I(\aa^*)$.
These elements are the images of $3\bdelta^2-2(\beps_1^2+\beps_2^2+\beps_3^2)$ and $-(\bdelta^2-\beps_1^2)(\bdelta^2-\beps_2^2)(\bdelta^2-\beps_3^2)$ which are in $I(\h^*)$ as $\bdelta\mid_{\aa}=0$. 
Since the generators of $I(\aa^*)$ are restrictions of elements in $I(\h^*)$,  the restriction map from $I(\h^*)$ to $I(\aa^*)$ is surjective.

\subsection{Case $\g=F_4, \mathfrak k=\mathfrak{sl}_2\oplus D(2,1;2)$.}
Here
\[I(\mathfrak{a}^{*})=\left\{ f\in\mathbb{C}\left[\eps_1^2,\eps_2^2,\eps_3^2\right]^{S_3}\ |\ \left(\frac{\partial}{\partial\eps_1}+\frac{\partial}{\partial\eps_2}+\frac{\partial}{\partial\eps_3}\right)f\in\left\langle \eps_1+\eps_2+\eps_3\right\rangle \right\},\]
(see Case \ref{f4_sl_d21a} in the Appendix).

Let us show that $I(\h^*)$ surjects onto $I(\aa^*)$.
This will in particular imply that $R_\theta$ is surjective.
Let $\hat{D}=\frac{\partial}{\partial\eps_1}+\frac{\partial}{\partial\eps_2}+\frac{\partial}{\partial\eps_3}-3\frac{\partial}{\partial\delta}$
and recall that $I(\h^*)$ consists of polynomials $\hat f$ in $\mathbb{C}\left[\beps_1^2,\beps_2^2,\beps_3^2,\bdelta^2\right]^{S_3}$ for which $ \hat D\hat f\in\left\langle \beps_1+\beps_2+\beps_3+\bdelta\right\rangle $ (here $\bdelta\mid_{\aa}=0$).
Let $f\in I(\mathfrak{a}^{*})_{2d}$. 
Then $\hat D f=p_1\cdot F$ for
some $F\in\mathbb{C}\left[\varepsilon_1,\varepsilon_2,\eps_3\right]_{2d-2}^{S_3}$. 

Let $p_n=\eps_1^n+\eps_2^n+\eps_3^n$. Recall that $p_3,p_2,p_1$ freely generate $\mathbb C\left[\eps_1,\eps_2,\eps_3\right]^{S_3}$ and that $p_6,p_4,p_2$ freely generate  $\mathbb C\left[\eps_1^2,\eps_2^2,\eps_3^2\right]^{S_3}$. Note that $\hat D p_n=np_{n-1}$.
Write 
\[\hat{f}:=f+\sum_{r_0+3r_6+2r_4+r_2=d}a_{\vec{r}}\delta^{2r_0}p_6^{r_6}p_4^{r_4}p_2^{r_2}\]
where $r_{0}\ge1$, $a_{\vec{r}}\in\mathbb{C}$. 
Then $\hat{f}\in I(\mathfrak{h}^{*})$ if and only if $\hat{D}\hat{f}\mid_{\delta=-p_1}=0$. This means that 
\[
\hat{D}\hat{f}\mid_{\delta=-p_1}=p_1F+\sum_{r_{0}+3r_6+2r_{4}+r_2=d}a_{\vec{r}}\hat{D}\left(\delta^{2r_{0}}p_6^{r_6}p_{4}^{r_{4}}p_2^{r_2}\right)\mid_{\delta=-p_1}=0.
\]
By the following technical lemma, one can find $a_{\vec r}$'s that will make the above expression zero for every $F\in \mathbb{C}\left[\eps_1,\eps_2,\eps_3\right]_{2d-2}^{S_3}$ which proves the existence of a preimage in $I(\h^*)$.

\begin{lemma}
The elements 
\[q_{\vec r}:=p_1^{-1}\cdot\hat D \left(\delta^{2r_0}p_6^{r_6}p_4^{r_4}p_2^{r_2}\right)\mid_{\delta=-p_1},\text{ where }r_0\ge1,r_0+3r_6+2r_4+r_2=d,\]
span $\mathbb{C}\left[\eps_1,\eps_2,\eps_3\right]_{2d-2}^{S_3}$.
\end{lemma}

\begin{proof}
We show that every basis element $p_3^{n_3}p_2^{n_2}p_1^{n_1}$ in $\mathbb{C}\left[\varepsilon_1,\varepsilon_2,\varepsilon_3\right]_{2d-2}^{S_3}$
appears as a leading monomial of some $q_{\vec r}$. 
Indeed,
\begin{align*}
\hat{D}\left(\delta^{2r_0}p_6^{r_6}p_4^{r_4}p_2^{r_2}\right) & =
-6r_0\delta^{2r_0-1}p_6^{r_6}p_4^{r_4}p_2^{r_2}+6r_6\delta^{2r_0}p_6^{r_6-1}p_{5}p_4^{r_4}p_2^{r_2}\\
 & +4r_4\delta^{2r_0}p_6^{r_6}p_4^{r_4-1}p_3p_2^{r_2}+2r_2\delta^{2r_0}p_6^{r_6}p_4^{r_4}p_2^{r_2-1}p_1.
\end{align*}
So 
\begin{align*}
q_{\vec r } & =6r_0p_6^{r_6}p_4^{r_4}p_2^{r_2}p_1^{2r_0-2}+6r_6p_6^{r_6-1}p_5 p_4^{r_4}p_2^{r_2}p_1^{2r_0-1}\\
 & +4r_4p_6^{r_6}p_4^{r_4-1}p_3 p_2^{r_2}p_1^{2r_0-1}+2r_2p_6^{r_6}p_4^{r_4}p_2^{r_2-1}p_1^{2r_0}.
\end{align*}
The following relations imply that the leading monomial of $q_{\vec r}$ is $p_3^{2r_{6}+r_{4}}p_2^{r_2}p_1^{r_{4}+2r_{0}-2}$.
\begin{align*}p_{6}&=\frac{1}{3}p_3^2+p_3p_2p_1+\frac{1}{3}p_3p_1^{3}+\frac{1}{4}p_2^{3}-\frac{3}{4}p_2^2p_1^2-\frac{1}{4}p_2p_1^{4}+\frac{1}{12}p_1^{6}\\
p_{5}&=\frac{5}{6}p_3p_2+\frac{5}{6}p_3p_1^2-\frac{5}{6}p_2p_1^{3}+\frac{1}{6}p_1^{5}\\
p_{4}&=\frac{4}{3}p_3p_1+\frac{1}{2}p_2^2-p_1^2p_2+\frac{1}{6}p_1^{4}.
\end{align*}
Thus, given $p_3^{n_3}p_2^{n_2}p_1^{n_1}$ with $3n_3+2n_2+n_1=2d-2$, it appears as a leading monomial for $q_{\vec r}$ for $\vec r=(r_0,r_6,r_4,r_2)=( \left\lfloor \frac{n_1+2}{2}\right\rfloor,\left\lfloor \frac{n_3}{2}\right\rfloor,p(n_3),n_2)$, where $p(n_3)\in\{0,1\}$ is the parity of $n_3$.
\end{proof}

The ring $I(\aa^*)$ is generated by the images of the generators of $I(\h^*)$, namely
\begin{align*}
&3(\eps_1^2+\eps_2^2+\eps_3^2), \\
&\eps_1^6+\eps_2^6+\eps_3^6+\sum_{i<j}(\eps_i-\eps_j)^6-\frac{1}{32}\sum(\eps_1\pm\eps_2\pm\eps_3)^6,\\
&\prod(\eps_1\pm\eps_2\pm\eps_3)^2 (\eps_1^{2k}+\eps_2^{2k}+\eps_3^{2k}),\  k\in\mathbb Z_{\ge 1}.
\end{align*}
\section{\red Appendix: \black  restricted roots systems for even symmetric pairs.}
In this appendix we list properties of restricted root systems for all cases that correspond to Iwasawa involutions. 
\subsection{Cartan Subspace}
We use the standard basis $\beps_1,\ldots,\beps_{\bar m},\bar\d_1,\ldots,\bar\d_{\bar n}$ for denoting the roots of $\g$. For $\g\ne D(2,1,a),F_4$, 
$$(\beps_i,\beps_j)=\d_{ij}=-(\bar\d_i,\bar\d_j),\quad (\beps_i,\bar\d_j)=0.$$
For $\g=D(2,1,a)$, we have
$(\beps_1,\beps_1)=-a-1$, $(\beps_2,\beps_2)=1$ and $(\beps_3,\beps_3)=a.$
For $\g=F_4$, we have
$(\beps_i,\beps_j)=2\d_{ij}$,\ $(\bdelta,\bdelta)=-6$ and $(\beps_i,\bdelta)=0$. Note that for  $\g=G_3$, we use the basis $\beps_1,\beps_2,\beps_3,\bdelta$ such that $\beps_1+\beps_2+\beps_3=0$, $(\beps_i,\beps_j)=-1$ for $i\ne j$ and $(\beps_i,\beps_i)=2$ (see for example \cite[10.9]{GK}).

We use the notation $\eps_1,\ldots,\eps_m,\d_1,\ldots,\d_n$ to denote a basis to $\aa^*$. The restriction function from $\h^*$ to $\aa^*$ is described in each case. 
We denote by $\bar e_1,\ldots,\bar e_{\bar m},\bar d_1,\ldots,\bar d_{\bar n}$ the dual basis to $\beps_1,\ldots,\beps_{\bar m},\bdelta_1,\ldots,\bdelta_{\bar n}$, and by $e_1,\ldots,e_m,d_1,\ldots,d_n$ the dual basis to $\eps_1,\ldots,\eps_m,\d_1,\dots,\d_n$.
\subsection{Restricted Root Systems.}\label{sec:restricted root systems}

Below is the list of type of the restricted root system following \cite{SV} and the restriction function from $\h^*$ to $\aa^*$.
\vspace{4mm}
\hfill\break
\def\arraystretch{1.2}
\begin{center}
    
\begin{tabular}{c||c|c|c}
Case & $\g$ &$\k$ & Type \\
\hline
\hline 
  
  \rownumber\label{gl_osp} & $ \mathfrak{gl}(m|2n)$ & $ \mathfrak{osp}(m|2n)$ &$A(m-1|n-1)$\\
  \hline
  
  \rownumber \label{gl_plus_gl} & $ \mathfrak{gl}(2m+a|2n+b)$ & $ \mathfrak{gl}(m|n)\oplus \mathfrak{gl}(m+a|n+b)$ &$BC(m|n)$ \\
  &$a,b\ge 0$&\\
  \hline
  
  \rownumber\label{osp_plus_osp} & $ \mathfrak{osp}(2m+a|4n+2b)$ & $ \mathfrak{osp}(m|2n)\oplus \mathfrak{osp}(m+a|2n+2b)$ &$BC(m|n)$\\
  &$a,b\ge 0$&\\
  \hline

  \rownumber\label{osp_gl_odd} & $ \mathfrak{osp}(4m+2|2n)$ & $ \mathfrak{gl}(2m+1|n)$  &$BC(m|n)$ \\
  \hline

  \rownumber\label{osp_gl_even} & $ \mathfrak{osp}(4m|2n)$ & $ \mathfrak{gl}(2m|n)$  &$BC(m|n)$ \\
  \hline
  
  \rownumber\label{f4_gosp24}  & $ F_4$ & $ \mathfrak{gosp}(2|4)$  &$BC(2|1)$ \\
  \hline

  \rownumber \label{f4_sl14} & $ F_4$ & $ \mathfrak{sl}(1|4)$& $BC(1|1)$\\
  \hline
  
  \rownumber\label{d21a_osp22} & $ D(2,1,a)$ & $ \mathfrak{osp}(2|2)\oplus \mathfrak{so}(2)$ &$BC(1|1)$\\
  \hline

  \rownumber\label{g3_d21a} & $ G_3$ & $ D(2,1,3)$ &$G_2$ \\
  \hline

  \rownumber\label{f4_sl_d21a} & $ F_4$ & $ \mathfrak{sl}_2\oplus D(2,1;2)$ & exotic\\
  \hline

\end{tabular}
\end{center}

\vspace{4mm}
We describe in each case the restriction function from $\h^*$ to $\aa^*$. The standard basis elements which are not written are being restricted to zero. In the following table $1\le i\le m$ and $1\le j \le n$.
\vspace{4mm}
\hfill\break
\begin{center}
\begin{tabular}{c||c c}
\
Case & \multicolumn{2}{c}{Restriction from $\h^*$ to $\aa^*$}\\
\hline\hline
    
    \ref*{gl_osp}        & $\eps_i=\beps_i\mid_\aa$          &
    $\delta_j=\bdelta_j\mid_\aa=\bdelta_{n+j}\mid_\aa$\\
    
    \ref*{gl_plus_gl}    & $\eps_i=\beps_i\mid_\aa=-\beps_{m+i}\mid_\aa$    & $\delta_j=\bdelta_j\mid_\aa=-\bdelta_{n_1+j}\mid_\aa$ \\ 
    
    \ref*{osp_plus_osp}  &$\eps_i=\beps_i\mid_\aa$       &
    $\delta_j=\bdelta_j\mid_\aa=-\bdelta_{n+j}\mid_\aa$ \\
    
    \ref*{osp_gl_odd}  &  $\eps_i=\beps_i\mid_\aa=-\beps_{2m-i+2}\mid_\aa$   &
    $\delta_j=\bdelta_j\mid_\aa$    \\
    
    \ref*{osp_gl_even}  &  $\eps_i=\beps_i\mid_\aa=-\beps_{2m-i+1}\mid_\aa$   &
    $\delta_j=\bdelta_j\mid_\aa$    \\

    \ref*{f4_gosp24}  & $\eps_1=\frac{\beps_1+\beps_2}{2}\mid_\aa$,\  $\eps_2=\frac{\beps_1-\beps_2}{2}\mid_\aa$      &
    $\d=-\frac{\bdelta}{2}\mid_\aa$ \\
    
    \ref*{f4_sl14}  & $\eps=\frac{\beps_1}{2}\mid_\aa$        &
    $\delta=\frac{\bdelta}{2}\mid_\aa$ \\
    
    \ref*{d21a_osp22}  & $\eps=\beps_1\mid_\aa$,      &$\d=-\beps_2\mid_\aa$
     \\
    
    \ref*{g3_d21a}  & $\eps_1=\beps_1\mid_\aa$,    $\eps_2=\beps_2\mid_\aa$,    
 $\eps_3=\beps_3\mid_\aa$      &
     \\
    
    \ref*{f4_sl_d21a}  & $\eps_1=\beps_1\mid_\aa$,    $\eps_2=\beps_2\mid_\aa$,    
 $\eps_3=\beps_3\mid_\aa$      &
     \\
\end{tabular}
\end{center}
\hfill\break

\vspace{4mm}
We write below the multiplicity of each restricted root. When the root space is pure, we write a positive number for the dimension of an even space and a negative number for the dimension of an odd space. 
We let $k:=-\frac{(\d_j,\d_j)}{(\eps_i,\eps_i)}$ be the deformation parameter used in Section \ref{subsec:D_alpha} to compute $D_\alpha$.

\subsubsection{Multiplicities in type $A(m-1|n-1)$}

\black 
\hfill\break

\begin{center}
    \begin{tabular}{c||c|c|c|c}
        Case            & $\pm(\eps_i-\eps_j)$      &$\pm(\d_i-\d_j)$   &$\pm(\eps_i-\d_j)$& $k$\\ \hline\hline
        \ref*{gl_osp}    &$1$        & $4$       &$-2$   & $\frac{1}{2}$ \\ \hline
        
    \end{tabular}
\end{center}
\hfill\break
{
To compute $k=-\frac{(\d_j,\d_j)}{(\eps_i,\eps_i)}$, we note that the dual basis element to $\eps_i$ in $\aa$ is $\bar e_i$. The dual basis element to $\d_j$ in $\aa$ is $\frac{\bar d_j+\bar d_{n+j}}{2}$. Hence $(\eps_i,\eps_i)=(\bar e_i,\bar e_i)=1$ and $(\d_j,\d_j)=\left(\frac{\bar d_j+\bar d_{n+j}}{2},\frac{\bar d_j+\bar d_{n+j}}{2} \right)=-\frac{1}{2}$.}

\subsubsection{Multiplicities in type $BC(m|n)$}

\hfill\break
\begin{center}
    
    \begin{tabular}{c||c|c|c|c|c|c|c|c}
        Case            &$\pm\eps_i\pm\eps_j$ &$\pm\eps_i$     &$\pm2\eps_i$&$\pm\d_i\pm\d_j$ &$\pm\d_i$       &$\pm2\d_i$  &$\pm\eps_i\pm\d_j$&$k$\\ \hline\hline
        
        \ref*{gl_plus_gl}& $2$       &$(2a|2b)$  &$1$    &$2$    &$(2b|2a)$  & $1$   &$-2$  &$1$  \\ \hline
        
    \ref*{osp_plus_osp}  &$1$        &$(a|2b)$   &       &$4$    &$(4b|2a)$   & $3$   &$-2$    &$\frac{1}{2}$\\ \hline
    
        \ref*{osp_gl_odd}    & $4$       & $2$       &$1$    &$1$    &$-2$       & $1$   &$-2$ &$2$    \\ \hline
    
        \ref*{osp_gl_even}    & $4$       &       &$1$    &$1$    &       & $1$   &$-2$    &$2$\\ \hline
        
        \ref*{f4_gosp24} & $3$       &        & $1$      &       &       &$1$        &$-2$ &$\frac{3}{2}$   \\ \hline
        
        \ref*{f4_sl14}   &           &       & $5$       &       &        &  $1$     &$-4$ &$3$   \\ \hline
        
        \ref*{d21a_osp22} &       &           &  $1$  &       &           &   $1$    & $-2$  &$\frac{1}{a+1}$     \\ \hline

    \end{tabular}
\end{center}
\hfill\break
Note that in Case \ref*{f4_gosp24}, $m=2, n=1$ and in Cases  \ref*{f4_sl14} and \ref*{d21a_osp22}, $m=n=1$.
The computation of $k$ form  is done in a similar fashion as Case \ref*{gl_osp}.\black 
\subsubsection{Multiplicities in type $G_2$}
\hfill\break
\begin{center}
    
    \begin{tabular}{c||c|c}
        Case            &long root & short root     \\ \hline\hline
       
        \ref*{g3_d21a}   &$1$        &$(1|2)$          \\ \hline

    \end{tabular}
\end{center}
\hfill\break

\subsubsection{An exotic case}

\hfill\break
\begin{center}
    
    \begin{tabular}{c||c|c|c}
        Case            &$\pm\eps_i\pm\eps_j$ &$\pm\eps_i$     &$\frac{1}{2}\left(\pm\eps_1\pm\eps_2\pm\eps_3\right)$\\ \hline\hline

        \ref*{f4_sl_d21a}&$1$        &$1$     &$-2$    \\ \hline

    \end{tabular}
\end{center}
\hfill\break

In this case $(\eps_i,\eps_i)=1$ and $(\eps_i,\eps_j)=0$ for $i\ne j$.


\subsection{Computation of $D_{h_\alpha}$}\label{subsec:D_alpha}
For Cases \ref{gl_osp}-\ref{g3_d21a}, we compute the derivative condition in the following manner. 
    Suppose that $\eps_1,\d_1\in\aa^*$ are such  that $(\eps_1,\eps_1)=1$, $(\d_1,\d_1)=-k$, $(\eps_1,\d_1)=0$ and $\alpha=\eps_1-\d_1$ is a former isotropic root. Let $e_1,d_1$ be the dual elements in $\aa$ to $\eps_1,\d_1$. Then $e_1+d_1\in\ker\a$ and $h_\alpha=e_1+\frac{1}{k}d_1\in(\ker\a)^\perp$.
    Hence the condition $D_{h_\a}f\in\ker\a$ becomes $\left(\frac{\partial}{\partial\eps_1}+\frac{\partial}{k\partial\d_1}\right)f\in\langle\eps_1-\d_1\rangle$.

    For Case \ref{f4_sl_d21a}, take $\a=\eps_1+\eps_2+\eps_3$. Then $\ker\a=\text{span}\{e_1-e_2,e_2-e_3\}$ and $h_\a=e_1+e_2+e_3\in(\ker\a)^\perp$. Hence the derivative condition becomes $(\frac{\partial}{\partial\eps_1}+\frac{\partial}{\partial\eps_2}+\frac{\partial}{\partial\eps_3})f\in\langle \eps_1+\eps_2+\eps_3\rangle$.

\subsection{Generators for $I(\aa^*)$}
\label{sec:generators table}
We list the generators in each case.    
\hfill\break
\begin{center}
    
\begin{tabular}{c||c }
\
Case & {Generators}\\
\hline\hline
    
    \ref*{gl_osp}        & $\eps_1^k+\ldots+\eps_{m}^{k}-2\delta_1^{k}-\dots-2\delta_{n}^{k}$     \\ \hline

    \ref*{gl_plus_gl}    & $\eps_1^{2k}+\ldots+\eps_{m}^{2k}-\delta_1^{2k}-\dots-\delta_{n}^{2k}$ \\ \hline
    
    \ref*{osp_plus_osp}, $a\ne 0$  &$\eps_1^{2k}+\ldots+\eps_{m}^{2k}-2\delta_1^{2k}-\dots-2\delta_{n}^{2k}$    \\ \hline
    
    
    \ref*{osp_gl_odd}, \ref*{osp_gl_even} &  $2\eps_1^{2k}+\ldots+2\eps_{m_1}^{2k}-\delta_1^{2k}-\dots-\delta_n^{2k}$   \\\hline

    \ref*{f4_gosp24}  & $\eps_1^{2k}+\eps_2^{2k}-\frac{2}{3}\d^{2k}$\\ \hline
    
    \ref*{f4_sl14}  & $3\eps^2-\d^2,\  3\eps^6-\d^6-16(\d^2-\eps^2)^3,\  (\d^2-\eps^2)^4 \eps^{2i}\d^{2j}$ \\ \hline
    
    \ref*{d21a_osp22}  & $\eps^{2k}-(a+1)\d^{2k}$     \\ \hline
    
    \ref*{g3_d21a}  & $\eps_1^2+\eps_2^2+\eps_3^2,\ \eps_1^2\eps_2^2\eps_3^2$   \\ \hline
    
    \ref*{f4_sl_d21a}  & $3(\eps_1^2+\eps_2^2+\eps_3^2)$,  \\
    & $\eps_1^6+\eps_2^6+\eps_3^6+\sum_{i<j}(\eps_i-\eps_j)^6-\frac{1}{32}\sum(\eps_1\pm\eps_2\pm\eps_3)^6$,\\
    & $\prod(\eps_1\pm\eps_2\pm\eps_3)^2 (\eps_1^{2k}+\eps_2^{2k}+\eps_3^{2k})$\\ \hline
\end{tabular}
\end{center}
\hfill\break
Here $k\in\mathbb Z_{\ge 1}$  and $i,j\in\mathbb Z_{\ge 0}$ in Case \ref*{f4_sl14}.

We remark that for the Case \ref{osp_plus_osp}, $a=0$, we have $I(\aa^*)=I_1(\aa^*)\oplus I_2(\aa^*)$ as shown in Section \ref{sec:hard D-Case}. 
The subring $I_1(\aa^*)$ is generated by $\eps_1^{2k}+\ldots+\eps_{m}^{2k}-\delta_1^{2k}-\dots-\delta_{n}^{2k}$, $k\in\mathbb Z_{\ge 1}$ and $I_2(\aa^*)$ is a module over $I_1(\aa^*)$.
However, we do not know the generators of $I_2(\aa^*)$ as an $I_1(\aa^*)$-module.


\end{document}